\documentclass{amsart}

\usepackage[T2C]{fontenc}

\usepackage[american]{babel}

\usepackage{amsfonts,amssymb,mathrsfs,amscd}
\usepackage{tabularx}
\usepackage{rotating}
\usepackage[all]{xy}
\usepackage{hypbmsec}
\usepackage{float}

\usepackage{graphicx}

\theoremstyle{plain}
\newtheorem{theorem}{Theorem}[section]
\newtheorem{lemma}{Lemma}[section]
\newtheorem{corollary}{Corollary}[section]
\newtheorem{proposition}{Proposition}[section]
\theoremstyle{definition}

\newtheorem{remark}{Remark}[section]
\newtheorem{example}{Example}[section]

\numberwithin{equation}{section}

\DeclareMathOperator{\intt}{int}

\DeclareMathOperator{\sign}{sign}

\DeclareMathOperator{\rint}{rint}

\newcommand*{\uR}{\underline{\mathbb{R}}}

\newcommand*{\oS}{\overline{S}}

\newcommand*{\RR}{{\mathbb{R}}}

\newcommand*{\cR}{\mathcal{R}}
\newcommand*{\cL}{\mathcal{L}}
\newcommand*{\cI}{\mathcal{I}}
\newcommand*{\cJ}{\mathcal{J}}

\newcommand*{\rr}{{\bf r}}
\newcommand*{\PPrr}{\mathcal{P}_{\bf r}}
\newcommand*{\PP}{\mathcal{P}}
\DeclareMathOperator{\con}{con}

\newcommand*{\zz}{\mathbf{z}}
\newcommand*{\vv}{\mathbf{v}}

\newcommand*{\mv}{\mathbf{m}}

\newcommand*{\ww}{\mathbf{w}}
\newcommand*{\mol}{\overline{m}}
\newcommand*{\mul}{\underline{m}}

\newcommand*{\yy}{\mathbf{y}}

\newcommand*{\xx}{\mathbf{x}}
\newcommand*{\NN}{\mathbb{N}}

\usepackage{enumitem}

\def\III_#1{{I_{#1}}}

\begin{document}

\title[Sum of translates]{On the weighted Bojanov--Chebyshev Problem and the sum of translates method of Fenton}

\author[B\'alint Farkas]{B.~Farkas}
\address{School of Mathematics and Natural Sciences, University of Wuppertal}
\email{farkas@math.uni-wuppertal.de}

\author[B\'ela Nagy]{B.~Nagy}
\address{Department of Analysis, Bolyai Institute, University of Szeged}
\email{nbela@math.u-szeged.hu}

\author[Szil\'ard Gy.~R\'ev\'esz]{Sz.\,Gy.~R\'ev\'esz}
\address{Alfr\'ed R\'enyi Institute of Mathematics}
\email{revesz@renyi.hu}

%




\begin{abstract}
Minimax and
maximin problems are investigated
for a special class of
functions on the interval
$[0,1]$.
These functions are
sums of translates of
positive multiples of
one kernel function and
a very general external field function.
Due to our very general setting
the obtained minimax, equioscillation, and characterization results extend
those of Bojanov, Fenton,
Hardin--Kendall--Saff and
Ambrus--Ball--Erd\'elyi.
Moreover, we discover a surprising intertwining phenomenon of interval maxima,
which provides new information even in the most classical extremal problem of Chebyshev.

Bibliography: 23-titles.
\end{abstract}


\maketitle
\section{Introduction}\label{ec:Intro}

Our starting point is the following theorem of Bojanov (see Theorem 1 in \cite{Bojanov1979}).
\begin{theorem}\label{thm:Bojanov}
For $n\in \NN$ let $\nu_1,\nu_2,\ldots,\nu_n>0$ be positive integers.
Given an interval $[a,b]$ ($a<b$) there exists a unique set of points $x_1^* \le \ldots\le x_n^*$ such that
\[
\left\|(x-x_1^*)^{\nu_1} \ldots (x-x_n^*)^{\nu_n} \right\|
\\=
\inf_{a\le x_1\le\ldots\le x_n\le b}
\left\|(x-x_1)^{\nu_1} \ldots (x-x_n)^{\nu_n} \right\|
\]
where $\|\cdot \|$ is the sup norm over $[a,b]$.
Moreover, $a<x_1^*<\ldots<x_n^*<b$, and the extremal polynomial $T(x):=(x-x_1^*)^{\nu_1} \ldots (x-x_n^*)^{\nu_n}$ is uniquely determined by the
 equioscillation property: There is $(t_k)_{k=0}^n$ with  $a=t_0<t_1<\ldots<t_n=b$ and
\[
T(t_k)=(-1)^{\nu_{k+1}+\ldots+\nu_n} \|T\| \qquad (k=0,1,\ldots,n)
\]
\end{theorem}

This contains the classical and well-investigated Chebyshev problem---where all the $\nu_j=1$.
Generalizations of the original
Chebyshev problem were extensively studied for
systems $\{\varphi_j\}_{j=1}^n$ with the
so-called Chebyshev- or Haar property
\cite{Haar} that all the
generated ``polynomials'' (linear combinations) $\sum_{j=1}^n \alpha_j \varphi_j$ have at most $n$ zeroes (on the interval $[a,b]$),
but direct adaptations of that approach abort very early as here for the Bojanov problem the occurring polynomials do not even form a vector space,
not to speak of the fact that the number of zeroes can be as large as $\nu_1+\dots+\nu_n$.

As in the Chebyshev problem it a natural to consider weighted maximum norms.
Note that $\|w f\|=C \Leftrightarrow - C \le f(x)w(x) \le C$
($x\in [a,b]$) with equality at some points.
Introduce $W(x):=1/w(x)$, consider
 $-CW(x) \le f(x) \le CW(x)$ ($x\in [a,b]$) and \emph{minimize  $C$},
needed for the validity of these bounds. If $\nu_1=\dots=\nu_n=1$, very general results are known, even for non-symmetric weights: $-C U(x) \le f(x) \le C V(x)$.
Extremal polynomials---called ``snake poly\-no\-mi\-als''---equioscillate between these bounds \cite{Karlin}.
For more on this  we refer the reader, e.g.,
to  \cite{Davydov}, \cite{Davydov2}, \cite{KarlinStudden}, \cite{Nikolov}.

We consider only equal lower and upper bounds.
This is natural from the point of view of our approach, which can deal with absolute values, but not with signs. However, in case of algebraic polynomials as above, it is easy to trace back the signs taken in various intervals between zeroes, and Bojanov's original results can easily be recovered\footnote{We will leave these to the reader throughout, though.}.
Let us only note here that our discussion will reveal that in the above Bojanov Theorem we can as well say that the unique system of optimizing nodes is characterized by the attainment, in all the intervals $[a,x_1^*], [x_1^*,x_2^*],\ldots,[x_{n-1}^*,x_n^*], [x_n^*,b]$, at some points $t_0,t_1,\ldots,t_n$, of the norm: $|T(t_i)|=\|T\|, (i=0,1,\ldots,n)$---so for characterization there is no need to assume \emph{signed equioscillation}, but the only thing what really matters is the attainment of the norm. 
Note that equioscillation properties are important in approximation theory, see the papers by Bojanov, Naidenov \cite{BojanovNaidenov1,BojanovNaidenov2}, Bojanov, Rahman \cite{BojanovRahman} and Nikolov, Shadrin \cite{NikolovShadrin, NikolovShadrin2}.

Our approach allows to consider arbitrary positive real exponents $\nu_i$, so to address the Bojanov-type extremal problem for so-called \emph{generalized algebraic polynomials}, cf. \cite{ErdelyiBorwein}. 
In fact the setting of the paper is  more general.
Let us point out that Bojanov in his papers used the classical, Chebyshev-Markov style approximation theoretic approach, 
with fine tracing of zeroes, monotonicity observations, interlacing properties, zero counting etc.
Bojanov's result was never extended to the weighted case or to trigonometric polynomials.

In \cite{TLMS2018} we used an approach similar to the one presented here to address general minimax problems on the torus, and obtained (generalized) trigonometric polynomial and also generalized algebraic polynomial versions of Bojanov's Theorem.
However, the approach there involved a pull-back of the interval to the torus, with nodes on the interval corresponding to \emph{a pair of symmetric trigonometric nodes} on the torus.
If the extremal situation was not symmetric, the backward transfer would not work. For this reason the method can at best be generalized to the weighted case if only we assume evenness of the weight also on the interval (if normalizing it to $[-1,1]$). 
However, our aim here is to obtain results also for not necessarily even weights, and the setting is thus different from the one in  \cite{TLMS2018}. The motivation for considering the torus case in \cite{TLMS2018} was the polarization problem, see also  \cite{AmbrusBallErdelyi} and \cite{HardinKendallSaff}. Here our applications are different.

The Bojanov-type minimization problem can be immediately rephrased to a minimax problem by taking logarithm:
$ \log\left| (x-x_1)^{\nu_1} \ldots (x-x_n)^{\nu_n}\right| = \sum_{j=1}^n \nu_j \log|x-x_j|$.
Thus the original multiplicative extremal problem is reformulated to an additive one, namely, to minimize (in $x_1,\ldots,x_n$ subject to $x_1\le \dots \le x_n$) the quantity $\max_{[a,b]} \sum_{j=1}^n \nu_j \log|x-x_j|$.
The weighted norm version is:
\[
\textbf{minimize} \  (\text{in} \  x_1 \le \ldots \le x_n) \quad \left\|(x-x_1)^{\nu_1} \ldots (x-x_n)^{\nu_n} w(x) \right\|_{C([a,b])}.
\]
After taking logarithm (and assuming that $w(x)\ge 0)$
the problem becomes
\[
\textbf{minimize}  \  (\text{in} \   x_1 \le \ldots \le x_n) \qquad \max_{[a,b]} \Bigl(
 \log w(x) + \sum_{i=1}^n \nu_i \log|x-x_i|\Bigr).
\]

We formulate just one model result of our investigations in a concrete situation.

\begin{theorem}\label{thm:BojanovGeneral} Let $n\in \NN$ and $r_1,r_2,\ldots,r_n$ be positive numbers, $[a,b]$ be a non-degenerate, compact interval, and $w$ be an upper semicontinuous, non-negative weight function on $[a,b]$, assuming non-zero values at least at $n$ interior points plus at least on one more point in $[a,b]$. Then there exists a unique extremizer set of points $a\le x_1^* \le \ldots\le x_n^*\le b$ such that
\[
\left\|w(x) |x-x_1^*|^{r_1}\ldots |x-x_n^*|^{r_n} \right\|
=
\inf_{a\le x_1\le\ldots\le x_n\le b}
\left\|w(x)|x-x_1|^{r_1} \ldots |x-x_n|^{r_n} \right\|,
\]
(with $\|\cdot \|$  the sup norm over $[a,b]$), and in fact $a<x_1^*<\ldots<x_n^*<b$.

Moreover, the extremal generalized polynomial $T(x):=\prod_{i=1}^n |x-x_i^*|^{r_i}$
is uniquely determined by the following equioscillation property:
There exists $(t_i)_{i=0}^n$,
interlacing with the $x_i^*$, i.e., $a\le t_0<x_1^*<t_1<x_2^*<\ldots<x_n^*<t_n\le b$ such that
\[
w(t_k)T(t_k)=\|wT\| \qquad (k=0,1,\ldots,n).\]
\end{theorem}

The real origins of the sum of translates approach go back to Fenton, see \cite{Fenton},
whose original aim was to prove a conjecture of P.D.~Barry from 1962 about the growth of entire functions, in which he succeeded in \cite{FentonEnt}. Even if it turned out that Barry's original problem had been already solved by Goldberg \cite{Goldberg} 
a little earlier, later on Fenton showed further fruitful applications of his approach, see \cite{FentonCos, FentonCos2}. Our main results extend and further Fenton's original work on the sum of translates function and related minimax problems. Here we do not discuss possible further applications in the theory of entire functions.

What we do explore somewhat are a few applications to approximation theory.
Apart from the above sketched Bojanov-direction, we show that some seemingly unrelated questions, like, e.g., Chebyshev constants of the union of $k$ intervals, can also be dealt with.
Here it bears a crucial relevance that the weights we allow here are only assumed to be upper semicontinuous, contrary to  \cite{TLMS2018} where logarithmic concavity was a fundamental assumption.

The  most interesting findings are in
Theorem \ref{thm:Intertwining}, because the phenomena described there seem not having been observed so far, not even for the most classical case of the original Chebyshev problem.
It is well-known that Chebyshev nodes have the property that for any other node system some of the arising interval maxima stay below, while some of the interval maxima becomes larger than the maxima of the Chebyshev polynomial (which, by equioscillation, is attained as interval maxima for all the $n+1$ intervals between neighboring nodes and the endpoints).
But we will prove, that this ``intertwining property'' of
interval maxima is not at all a unique feature of the extremal Chebyshev nodes---in fact, \emph{any two node systems} exhibit this intertwining property with each other.


\section{The basic setting}\label{sec:basics}

A function $K:(-1,0)\cup (0,1)\to \RR$ will be called a \emph{kernel function}\footnote{The terminology used by Fenton in \cite{Fenton} is that $K$ is a \emph{cusp}, perhaps better fitting to his settings where $K$ is not assumed to satisfy the singularity condition \eqref{cond:infty} below, but rather the ``derivative singularity'' condition $\lim_{t\to 0\pm 0} K'(t)=\pm \infty$.} if it is concave on $(-1,0)$ and on $(0,1)$, and if it satisfies
\begin{equation}\label{eq:Kzero}
\lim_{t\downarrow 0} K(t) =\lim_{t\uparrow 0} K(t).
\end{equation}
These limits exist by the concavity assumption, and a kernel function has one-sided limits also at $-1$ and $1$. We set
\begin{equation*}
K(0):=\lim_{t\to 0}K(t),\quad K(-1):=\lim_{t\downarrow -1} K(t) \quad\text{and}\quad K(1):=\lim_{t\uparrow 1} K(t).
\end{equation*}
Note explicitly that we thus obtain the extended continuous function
$K:[-1,1]\to \RR\cup\{-\infty\}=:\uR$, and that we still have $\sup K<\infty$.
Note that a kernel function is almost everywhere differentiable.

We say that the kernel function $K$ is \emph{strictly concave} if it is strictly concave on both of the intervals $(-1,0)$ and $(0,1)$.

Further, we call it \emph{monotone}\footnote{These conditions---and more, like $C^2$ smoothness and strictly negative second derivatives---were assumed on the kernel functions in the ground-breaking paper of Fenton \cite{Fenton}.} if
\begin{equation}
\label{cond:monotone}\tag{M}
K \text{ is monotone decreasing on } (-1,0) \text{ and increasing on } (0,1).
\end{equation}
By concavity, under the monotonicity condition \eqref{cond:monotone} the values $K(-1)$ and $ K(1)$ are also finite. If $K$ is strictly concave, then \eqref{cond:monotone} implies \emph{strict monotonicity}:
\begin{equation}
\label{cond:smonotone}\tag{SM}
K \text{ is strictly decreasing on } [-1,0) \text{ and strictly increasing on } (0,1].
\end{equation}

A kernel function $K$ will be called \emph{singular} if
\begin{equation}\label{cond:infty}
\tag{$\infty$}
K(0)=-\infty.
\end{equation}
This condition is fundamental in the intended applications, so generally in this paper we will confine ourselves to singular kernels.

Let $n\in \NN=\{1,2,\dots,\}$ be fixed. We will call a function $J:[0,1]\to\uR$ an \emph{external $n$-field function}\footnote{Again, the terminology of kernels and fields came to our mind by analogy, which in case of the logarithmic kernel $K(t):=\log|t|$ and an external field $J(t)$ arising from a weight $w(t):=\exp(J(t))$ are indeed discussed in logarithmic potential theory. However, in our analysis no further potential theoretic notions and tools will be applied.} 
 or---if the value of $n$ is unambiguous from the context---simply a \emph{field function}, if it is bounded above  on $[0,1]$, and it takes finite values at more than  $n$ different points, where we count the points $0$ and $1$ with weight\footnote{The weighted counting makes a difference only for the case when $J^{-1}(\{-\infty\})$ contains the two endpoints; with only $n-1$ further interior points in $(0,1)$ the weights in this configuration add up to $n$ only, hence the node system is considered inadmissible.} $1/2$ only, while the points in $(0,1)$ are accounted for with weight $1$. Therefore, for a field function\footnote{To keep the coherence with our companion paper \cite{Homeo}, we do not assume a priori that a field function must be upper semicontinuous. However, in this paper this light extra assumption will be needed throughout---hence it will be signaled in the formulation of all the assertions relying also on this additional assumption.} $J$ the set  $(0,1)\setminus J^{-1}(\{-\infty\})$ has at least $n$ elements, and if it has precisely $n$ elements, then either $J(0)$ or $J(1)$ is finite.

\medskip

Further, we consider the \emph{open simplex}
\[
S:=S_n:=\{\yy : \yy=(y_1,\dots,y_n)\in (0,1)^n,\: 0< y_1<\cdots <y_n<1\},
\]
and its closure the \emph{closed simplex}
\[
\overline{S}:=\{\yy: \yy\in [0,1]^n,\: 0\leq y_1\leq \cdots \leq y_n\leq 1\}.
\]

\medskip
For any given $n\in \NN$, kernel function $K$,
constants $r_j>0$,  $j=1,\ldots,n$,
and a given field function $J$ we will consider the \emph{pure sum of translates function}
\begin{equation}\label{eq:puresum}
f(\yy,t):=\sum_{j=1}^n r_j K(t-y_j)\quad (\yy\in \overline{S},\: t\in [0,1]),
\end{equation}
and also the \emph{(weighted) sum of translates function}
\begin{equation}\label{eq:Fsum}
F(\yy,t):=J(t)+\sum_{j=1}^n r_jK(t-y_j)\quad (\yy\in \overline{S},\: t\in [0,1]).
\end{equation}

Note that the functions $J, K$ can take the value $-\infty$,
but not $+\infty$, therefore the sum of translates functions can be defined meaningfully.
Furthermore, if $g,h :A \to \uR$ are extended continuous functions
on some topological space $A$,
then their sum is extended continuous, too; therefore,  $f:\oS \times [0,1] \to \uR$ is extended continuous.
Note that for any $\yy\in \oS$ the function $f(\yy,\cdot)$ is finite valued on $(0,1)\setminus \{y_1,\dots, y_n\}$.
Moreover,  $f(\yy,0)=-\infty$ can happen only if some $y_j=0$
(hence also $y_1=0$) and $K(0)=-\infty$ or if some $y_j=1$ (and whence $y_n=1$ for sure) and $K_j(-1)=-\infty$.
Analogous statement can be made about the equality $f(\yy,1)=-\infty$.
By the assumption\footnote{Note that our somewhat complicated-looking assumptions 
on the weighted counting of points of finiteness of $J$ is \emph{the exact condition} 
to ensure this independently of the concrete choice of the kernels in general.}
  on $J$, we have that $F(\yy,\cdot)\not\equiv-\infty$, i.e., $\sup_{t\in[0,1]}F(\yy,t)>-\infty$.

Further, for any fixed $\yy \in \oS$ and $t\ne y_1,\ldots,y_n$
there exists a relative (with respect to $\oS$) open neighborhood of $\yy \in \oS$ where $f(\cdot,t)$ is concave (hence continuous).
Indeed, such a neighborhood is
$B(\yy,\delta):=\{\xx\in \oS: \|\xx-\yy\|<\delta\}$ with
$\delta:=\min_{j=1,\ldots,n} |t-y_j|$, where $\|\vv\|:=\max_{j=1,\ldots,n} |v_j|$.

We introduce the \emph{singularity set} of the field function $J$ as
\begin{equation}\label{eq:Xdef}
X:=X_J:=\{t\in [0,1]:\  J (t)=-\infty\},
\end{equation}
and recall that  $X^c:=[0,1]\setminus X$ has cardinality exceeding $n$
(in the above described, weighted sense), in particular $X\neq [0,1]$.

Writing $y_0:=0$ and $y_{n+1}:=1$  we also set for each $\yy\in \overline{S}$ and $j\in \{0,1,\dots, n\}$
\[
\III_j(\yy):=[y_j,y_{j+1}], \qquad m_j(\yy):=\sup_{t\in \III_j(\yy)} F(\yy,t),
\]
and
\begin{align*}
\mol(\yy)&:=\max_{j=0,\dots,n} m_j(\yy)=\sup_{t\in [0,1]}F(\yy,t),\quad\quad \mul(\yy):=\min_{j=0,\dots,n} m_j(\yy).
\end{align*}
As has been said above, for each $\yy\in \oS$ we have that $\mol(\yy)=\sup_{t \in [0,1]} F(\yy,t) \in \RR$ is finite.
On the other hand,  $m_j(\yy)=-\infty$ if and only if $F(\yy,\cdot)|_{I_{j}(\yy)}\equiv -\infty$, 
in which case $I_{j}(\yy)$ is called a singular interval (for the given node system). 
If there is $j\in \{0,1,\dots,n\}$ with $m_j(\yy)=-\infty$, then $\yy$ is called a \emph{singular} node system.
A node system $\yy\in \partial S= \oS\setminus S$ is called \emph{degenerate}.  
If the kernels are singular, then each degenerate node system is singular, and, 
furthermore, for a non-degenerate node system $\yy$ we have $m_j(\yy)=-\infty$ if and only if 
$\rint \III_j(\yy)\subseteq X$, where $\rint$ denotes the relative interior of a set with respect to $[0,1]$.

If $K$ is singular, then the functions $m_j$ are extended continuous. 
This is not immediately obvious due to the arbitrariness of $J$, but is proven in \cite{Homeo} as 
Lemma 3.3. We record this fact here explicitly for later reference.

\begin{proposition}\label{prop:cont}
Let $K$ be a singular kernel function and $J$ be an arbitrary $n$-field function. Then for each $j\in \{0,1,\dots, n\}$ the function
\[
m_j:\overline{S}\to \uR
\]
is continuous (in the extended sense).
Moreover, $\mol,\mul:\oS\to \RR$ are extended continuous and $\mol$ is finite valued and continuous in the usual sense.
\end{proposition}

\begin{remark}\label{rem:degen-not-eqiosci-gen}
For a singular kernel $K$, for an arbitrary $n$-field function $J$, for any node system $\ww\in\oS$ and $k\in \{0,\dots,n\}$, if $\III_k(\ww)$ is degenerate or singular,
then $m_k(\ww)=-\infty$, hence  $m_k(\ww)<\mol(\ww)$.
\end{remark}
We will be primarily interested in the minimax and maximin expressions
\begin{equation}\label{eq:minmaxsol}
M(S):=\inf_{\xx\in  S}\mol(\xx)=\inf_{\xx\in  S}\max_{j=0,1,\ldots,n}m_j(\xx),
\end{equation}
and
\begin{equation}\label{eq:maxminsol}
m(S):=\sup_{\xx\in S}\mul(\xx)=\sup_{\xx\in S}\min_{j=0,1,\ldots,n}m_j(\xx).
\end{equation}
In this respect an essential role is played by the \emph{regularity set}
\begin{align}\label{eq:Ydef}
	Y:=Y_n&=\{\yy\in S: \text{$m_j(\yy)\neq-\infty$ for $j=0,1,\dots,n$}\}.
\end{align}
Since $J$ is an $n$-field function necessarily it holds $Y\neq \emptyset$,
and we trivially have
\begin{equation*}
m(Y):=\sup_{\xx\in Y}\mul(\xx)=\sup_{\xx\in S}\mul(\xx)
=m(S)
=\sup_{\xx\in \oS}\mul(\xx)
=:m(\oS).
\end{equation*}
It will turn out, as a  byproduct of our results (see Theorem \ref{thm:EquiThm}), that also
\begin{equation*}
M(Y):=\inf_{\xx\in Y}\mol(\xx)=\inf_{\xx\in S}\mol(\xx)=M(S)
=\inf_{\xx\in \oS}\mol(\xx)
=:M(\oS).
\end{equation*}
holds.
As a matter of fact, for a singular,
strictly concave kernel function $K$ and
an upper semicontinuous $n$-field
function $J$ there is a unique $\ww\in \oS$
such that $\mol(\ww)=M(S)$,
in fact $\ww\in Y$ and
it is the unique  node system in $\oS$ with $\mul(\ww)=m(S)$,
see Corollary \ref{cor:minimaxmaximin-infty}.
Before we can prove these facts,
some further preparation is needed.

If the kernel $K$ is singular {(a main assumption in this paper)}, then we have
\begin{equation}\label{eq:Ydef-sing}
Y=\{\yy\in S: \rint \III_j(\yy)\not\subseteq X \text{ for } j=0,1,\dots,n\}.
\end{equation}
In case the kernel $K$ is singular,
it follows from \eqref{eq:Ydef-sing}
that the regularity set is an open subset of $S$, and
we have $S=Y$
if and only if $X$ has empty interior.

We also introduce the \emph{interval maxima vector function}
\[
\mv(\ww):=(m_0(\ww),m_1(\ww),\ldots,m_n(\ww)) \in \uR^{n+1} \quad (\ww \in \oS)
\]
and the \emph{interval maxima difference function} or simply \emph{difference function}
\begin{align}\label{eq:diffdefi} \notag
\Phi(\ww)& :=(m_1(\ww)-m_0(\ww), m_2(\ww)-m_1(\ww),\ldots, m_n(\ww)-m_{n-1}(\ww) )
\\ & =:(\Phi_1(\ww),\ldots,\Phi_n(\ww)) \in [-\infty,\infty]^{n}\quad (\ww\in Y).
 \end{align}
Note that $\mv:\oS\to \uR^{n+1}$ is extended continuous,
hence also $\Phi:Y\to \RR^n$ is a
continuous functions, by Proposition \ref{prop:cont}.

A key foothold for our  investigation below is the next (very special case of the main) result of \cite{Homeo}, see Corollary 2.2 therein.

\begin{theorem}\label{thm:homeo3}
For $n\in \NN$ let $r_1,\dots, r_n>0$, suppose that the singular kernel function $K$ is
strictly monotone \eqref{cond:smonotone} and take an arbitrary $n$-field function $J$.
Consider the sum of translates function $F$ as in \eqref{eq:Fsum}.

Then the corresponding difference function \eqref{eq:diffdefi}, restricted to $Y$, 	is a locally bi-Lipschitz homeomorphism between $Y$ and $\RR^n$.
\end{theorem}

Note that the theorem contains, among other things, the already non-trivial fact that $Y\subseteq S$ must be a (simply) connected domain.

\section{Perturbation lemmas}\label{sec:Perturb}

	A variant of the next lemma is contained in \cite{TLMS2018} (see Lemma 11.5; but also \cite{Rankin}, Lemma 10 on p.~1069). A similar but slightly simpler form was given by Fenton in \cite{Fenton} (though not formulated there explicitly, see around formula (15) in \cite{Fenton}).

\begin{lemma}[\textbf{Interval perturbation lemma}]\label{lem:widening}
Let $K$ be a kernel function. Let $0<\alpha<a<b<\beta<1$ and $p, q >0$.
Set
\begin{equation}\label{eq:mudef}
\mu:=\frac{p(a-\alpha)}{q(\beta-b)}.
\end{equation}
\begin{enumerate}[label={(\alph*)}]
\item
\label{mudef:parta}
If $K$ satisfies \eqref{cond:monotone} and $\mu\geq 1$, then for every $t\in [0,\alpha]$ we have
\begin{equation}\label{eq:wideninglemma}
pK(t-\alpha)+qK(t-\beta) \le pK(t-a)+qK(t-b).
\end{equation}
\item
\label{mudef:partb}
If $K$ satisfies \eqref{cond:monotone} and $\mu\leq 1$,
then \eqref{eq:wideninglemma} holds for every $t\in [\beta,1]$.
\item
\label{mudef:partc}
If $\mu=1$, even when $K$ does not satisfy \eqref{cond:monotone}, then \eqref{eq:wideninglemma} holds for every $t\in [0,\alpha] \cup [\beta,1]$.
\item
\label{mudef:partd}
In case of a strictly concave kernel function
\ref{mudef:parta}, \ref{mudef:partb}, and
\ref{mudef:partc}
hold with strict inequality in \eqref{eq:wideninglemma}.
\item
\label{mudef:parte}
If $K$ satisfies \eqref{cond:monotone}, then for every $t\in [a,b]$
\begin{equation}\label{eq:wideninginside}
pK(t-\alpha)+qK(t-\beta) \geq  pK(t-a)+qK(t-b),
\end{equation}
with strict inequality if $K$ is strictly monotone.
\end{enumerate}
\end{lemma}
\begin{proof} Rearranging \eqref{eq:wideninglemma} and
dividing by $q(\beta-b)$ yields the equivalent assertion
\[
\frac{p(a-\alpha)}{q(\beta-b)} \frac{K(t-\alpha)-K(t-a)}{a-\alpha} \le \frac{K(t-b)-K(t-\beta)}{\beta-b}.
\]
This expresses the inequality $\mu c \le C$ with $\mu$ defined in \eqref{eq:mudef} and
\[
c:=\frac{K(t-\alpha)-K(t-a)}{a-\alpha}, \qquad C:=\frac{K(t-b)-K(t-\beta)}{\beta-b}
\]
being the slopes of the chords of the graph of the kernel
function $K$ raised above the points $t-a, t-\alpha$ and $t-\beta, t-b$,
respectively.
Note that $t-\beta < t-b < t-a < t-\alpha$, and all of them are in $[0,1]$ if $t\in [\beta,1]$,
while all of them are in $[-1,0]$ if $t\in[0,\alpha]$.
It follows that these points lie in the same interval of concavity of $K$,
and the slope of the previous chord exceeds that of the chord to the right from it:
that is, we have $c \le C$.
In particular, for $\mu=1$
\ref{mudef:partc}
follows immediately,
even with strict inequality provided that $K$ is strictly concave,
as then even $c<C$ holds.

It remains to see when we may have $\mu c \le C$,
to which it suffices to see $\mu c \le c$.

Now if $c\le 0$ then this holds with $\mu\ge 1$, and if $c\ge 0$ then it holds with $\mu\le 1$.
In view of monotonicity, however, we surely have $c\le 0$ whenever $t-a, t-\alpha \le 0$, i.e.,
when $t\in[0,\alpha]$, so that in this case $\mu\ge 1$ suffices;
and the same way for $t\in[\beta,1]$ we have $t-a, t-\alpha >0$ and $c\ge 0$,
so that $\mu\le 1$ suffices.
Altogether we obtain both assertions
\ref{mudef:parta} and \ref{mudef:partb}.

Similar arguments yield the strict inequalities in
\ref{mudef:parta} and \ref{mudef:partb},
whenever $K$ is strictly concave. Altogether
\ref{mudef:partd}
is proved.

Assertion \ref{mudef:parte}
is immediate, since under the condition \eqref{cond:monotone}
for $t\in [a,b]$ we have $K(t-\alpha)\ge K(t-a)$ and also $K(t-\beta) \ge K(t-b)$, with strict inequality whenever $K$ is strictly monotone.
\end{proof}

Let us record the following, trivial but extremely useful fact as a separate lemma.

\begin{lemma}[\textbf{Trivial lemma}]\label{lem:trivi} Let $f, g, h :D \to \uR$ be  upper semicontinuous functions on some Hausdorff topological space $D$ and let $\emptyset\neq A \subseteq B \subseteq D$ be arbitrary. Assume
\begin{equation}\label{eq:trivia}
f(t)<g(t) \qquad\text{for all $t\in A$}.
\end{equation}
If $A\subseteq B$ is a compact set, then
\begin{equation}\label{eq:trivia-plus}
\max_A (f+h) < \sup_B (g+h) \qquad {\rm unless} \qquad h\equiv-\infty \quad\text{on }\quad A.
\end{equation}
\end{lemma}

\begin{proof} It is obvious that $\sup_A (f+h) \le \sup_B (g+h)$.
If $A$ is compact, $f+h$ attains its supremum at some point $a\in A$. If $h(a)=-\infty$, then also $\max_A (f+h)=f(a)+h(a)=-\infty$, and the strict inequality in \eqref{eq:trivia-plus} follows, unless $h+g \equiv -\infty$ all over $B$, hence in particular all over $A$. In this case, however, we must have $h\equiv -\infty$ all over $A$, since the strict inequality in the condition \eqref{eq:trivia} entails that $g>-\infty$ on $A$. Therefore the statement \eqref{eq:trivia-plus} is proved whenever $h(a)=-\infty$.
Now, if $h(a)>-\infty$ is finite, then we necessarily have $\max_A (f+h)=f(a)+h(a)<g(a)+h(a)\le \sup_B (g+h)$, and we are done also in this case. The proof is complete.
\end{proof}

\begin{remark}
The upper semicontinuity of the field function $J$ is needed
for the validity of the previous lemma,
on which our arguments rely heavily.
This is why we need to assume this
upper semicontinuity  in the main results.
\end{remark}

\begin{theorem}[\textbf{Minimax equioscillation}]\label{thm:EquiThm}
Let $n\in \NN$, let $K$ be a singular, strictly concave and strictly monotone  \eqref{cond:smonotone} kernel function and let $J:[0,1]\to \uR$ be an upper semicontinuous $n$-field function.

For $j=1,\dots,n$ consider $r_j>0$ and the sum of translates function as in \eqref{eq:Fsum}.

Then there is a minimum point $\ww\in \overline{S}$ of $\mol$ in $\oS$ (a minimax point), i.e.,
\begin{equation}\label{eq:MSpoint}
M(\overline{S})=
\inf_{\xx\in \overline{S}} \mol(\xx)=\inf_{\xx\in \overline{S}} \max_{j=0,1,\ldots,n}m_j(\xx)=\mol(\ww):=\max_{j=0,1,\ldots,n}m_j(\ww).
\end{equation}
Each such minimum point $\ww\in \overline{S}$ is an \emph{equioscillation point}, i.e., it satisfies
\begin{equation}\label{eq:equioscillation}
m_{0}(\ww)=m_{1}(\ww)=\ldots=m_{n}(\ww).
\end{equation}
Furthermore, the point $\ww$ is non-singular, i.e., belongs to the regularity set $Y$, so in particular to the open simplex $S$.
\end{theorem}

\begin{proof}
By continuity of $\mol$, see Proposition \ref{prop:cont}, some minimum point on the compact set $\oS$ must exist. Let $\ww\in\oS$ be any such minimum point. In the following we set first to prove that $\ww$ is an equioscillation point, i.e., $m_j(\ww)=\mol(\ww)$ for $j=0,1,\ldots,n$. Assume for a contradiction that $m_j(\ww)<\mol(\ww)$ for some $j\in \{0,1,\dots,n\}$.

\medskip
\textbf{Case 1.} First we consider the case when $\III_j=[w_j,w_{j+1}] \subseteq (0,1)$, and note that then $0<j<n$, since $w_0=0<w_j\le w_{j+1} <1=w_{n+1}$.
Consider the following sets of indices with positions of the $w_i$ at the left resp.{} right endpoint of $\III_j$:
\[
\cL:=\{ i \le j :\  w_i=w_j \}, \qquad \cR:=\{ i \ge j+1 \ :\  w_i=w_{j+1} \}.
\]
Note that in principle $\III_j$ can be degenerate, i.e.,{} $w_j=w_{j+1}$ can hold, but we have defined the index sets $\cL, \cR$ as disjoint. Further, we set
\[
L:=\sum_{i\in \cL} r_i, \qquad R:=\sum_{i\in \cR} r_i, \qquad \text{and}\qquad \cI:=\{0,\ldots,n+1\} \setminus (\cL\cup \cR).
\]
We apply Lemma \ref{lem:widening} with $\alpha:=w_j-Rh, a:=w_j, b:=w_{j+1}, \beta:=w_{j+1}+Lh, p:=L$ and $q:=R$, with a small, but positive $h>0$, to be specified suitably later on. As now the value of $\mu$ in \eqref{eq:mudef} is exactly 1, the mentioned lemma yields the strict inequality
\begin{equation}\label{eq:wideperturb}
L K(t-(w_j-Rh)) + R K(t-(w_{j+1}+Lh)) < L K(t-w_j) + R K(t-w_{j+1})
\end{equation}
for all points $t \in A :=[0,w_j-Rh] \cup [w_{j+1}+Lh,1]$.

Next, we define a new node system $\ww'$ by $w_i':=w_i-Rh =w_j-Rh=\alpha$ for all $i\in \cL$, $w'_i:=w_i+Lh=w_{j+1}+Lh=\beta$ for all $i\in\cR$, and the rest  unchanged: $w'_i:=w_i$ for $i\in \cI$. Note that if $h$ is smaller than the distance $\rho>0$ of the sets $\{w_i\ :\ i\in \cI\}$ and $\III_j$, then $0=w'_0 \le \dots \le w'_j < w'_{j+1} \le \dots \le 1$ and hence $\ww'\in \oS$. So assume $0<h<\rho$ from now on, and also that $h$ is chosen so small that $m_j(\ww')<\mol(\ww)$ remains in effect (continuity of $m_j$, see Proposition \ref{prop:cont}).

With the new node system $\ww'$ we have  $A=[0,1]\setminus \intt \III_j(\ww')$ and, moreover, inequality \eqref{eq:wideperturb} can be rewritten as
\[
L K(t-w'_j) + R K(t-w'_{j+1}) < L K(t-w_j) + R K(t-w_{j+1}) \quad \text{for } t \in A.
\]
Note that adding $J(t) + \sum_{i\in \cI} r_i K(t-w_i)$ to both sides, the left-hand side becomes $F(\ww',t)$, and the right-hand side becomes $F(\ww,t)$.

Applying the Trivial Lemma \ref{lem:trivi} with $A=[0,1]\setminus \intt \III_j(\ww')$, $B:=[0,1]\setminus \intt \III_j(\ww)$, and $D=[0,1]$ we obtain that $\max_A F(\ww',\cdot) < \max_B F(\ww,\cdot) \le \mol(\ww)$, unless the added expression $J(t) + \sum_{i\in \cI} r_i K(t-w_i)$ is identically $-\infty$ on $A$, in which case the left-hand side $\max_A F(\ww',\cdot)$ is also $-\infty$. In either case we are led to $\max_A F(\ww',\cdot)<\mol(\ww)$.

Taking into account $\max_{\III_j(\ww')} F(\ww',\cdot)=m_j(\ww')<\mol(\ww)$, we thus infer $\mol(\ww')<\mol(\ww)$, which contradicts the choice of $\ww$ as a minimum point of $\mol$. This contradiction proves that $m_j(\ww)=\mol(\ww)$ must hold.

\medskip
\textbf{Case 2.} Let now $0=w_j$. Then $\III_0=[0,w_1]\subseteq [0,w_{j+1}]=\III_j$, hence $m_0(\ww)\le m_j(\ww) <\mol(\ww)$. This implies that $\III_0(\ww)\subseteq [0,1)$ i.e.,{} $w_1<1$, so there is a maximal index $k\le n$ with $w_1=w_k$ and $w_k<w_{k+1}$. (Note that $w_1=\dots=w_k$ may or may not be in the position $0$---this does not matter.)

We consider the new node system $\ww'$ with $w_1'=\dots=w_k'=w_1+h$ (and the rest unchanged). With $0<h<w_{k+1}-w_k$ the new node system $\ww'$ also belongs to $\oS$. As above, for small enough $h$ continuity (Proposition \ref{prop:cont}) furnishes $m_0(\ww')<\mol(\ww)$.

Let now $t\in A:=[w_1+h,1]=[0,1]\setminus \rint \III_0(\ww')$. Taking into account the strict monotonicity condition \eqref{cond:smonotone} and $h>0$, we must have
\[
\sum_{i=1}^k r_i K(t-w_1-h) < \sum_{i=1}^k r_i  K(t-w_1) \qquad \text{for all } t \in A.
\]
Note that the left-hand side may attain $-\infty$ (at $t=w_1+h$), but not the right-hand side for $h>0$. If we add here $J(t)+\sum_{i=k+1}^n r_i K(t-w_i)$ to both sides, then the left-hand side becomes $F(\ww',t)$, and the right-hand side becomes  $F(\ww,t)$. Putting $B:=[w_1,1]=[0,1]\setminus \rint \III_1(\ww)$ an application of the Trivial Lemma \ref{lem:trivi} furnishes $\max_A F(\ww',\cdot) <\max_B F(\ww,\cdot) \le \mol(\ww)$, unless the left-hand side is $-\infty$; in particular, as $\mol(\ww)>-\infty$, in either case we obtain $\max_A F(\ww',\cdot) <\mol(\ww)$.

As above, taking into account $\max_{\III_1(\ww')} F(\ww',\cdot)=m_1(\ww')<\mol(\ww)$, we thus infer $\mol(\ww')<\mol(\ww)$, which is a contradiction with the minimality of $\mol(\ww)$. Therefore, $m_0(\ww)=m_j(\ww)=\mol(\ww)$ follows.

\medskip
\textbf{Case 3.} $w_{j+1}=1$. It is completely analogous to Case 2.

\medskip\noindent Cases 1-3 altogether yield that $\ww$ is an equioscillation point,
as claimed in \eqref{eq:equioscillation}.

\medskip\noindent
Furthermore, if $\III_k(\ww)$ is degenerate, or singular,
then by Remark \ref{rem:degen-not-eqiosci-gen} $m_k(\ww)=-\infty < \mol(\ww)$ must hold.
This contradicts equioscillation, hence is excluded by the above.
That is, $\ww \in Y$, i.e., no interval $\III_k(\ww)$ can be singular.
The proof is complete.
\end{proof}

There exist some maximum points of $\mul$ on $\oS$ by continuity and compactness.
Completely analogously to the above,  we can as well prove the following about these.

\begin{theorem}[\textbf{Maximin equioscillation}]\label{thm:maximinthm} Let $n\in \NN$, let $K$ be a singular, strictly concave and strictly monotone  \eqref{cond:smonotone} kernel function and let $J:[0,1]\to \uR$ be an upper semicontinuous $n$-field function.
For $j=1,\dots,n$ let  $r_j>0$ and  consider the sum of translates function as in \eqref{eq:Fsum}.
If $\yy\in\oS$ is a maximum point of $\mul$ (a maximin point), i.e.,
$\mul(\yy)=m(\oS):=\max_{\oS} \mul$,
then $\yy \in Y\subseteq S$, and $\yy$ is an equioscillation point.
\end{theorem}
\begin{proof}
By Theorem \ref{thm:EquiThm} we have a minimax point $\ww\in Y$. Thus for any maximin point $\yy\in\oS$ we have $\mul(\yy)=m(S)\ge\mul(\ww)=\mol(\ww)=\max_{[0,1]} F(\ww,\cdot) >-\infty$, hence for all $j=0,1,\ldots,n$ also $m_j(\yy)>-\infty$. We conclude $\yy\in Y$. By Remark \ref{rem:degen-not-eqiosci-gen} no degenerate intervals may exist among the $\III_j(\yy)$. This somewhat simplifies our considerations as compared to the proof of Theorem \ref{thm:EquiThm}.

\medskip It remains to prove that once $\yy\in \oS$ is a maximin point, we necessarily have that it is an equioscillation point, i.e., $m_j(\yy)=\mul(\yy)$ for $j=0,1,\ldots,n$. For the proof we assume for a contradiction that there exists some $j\in \{0,1,\dots,n\}$ with $m_j(\yy)>\mul(\yy)$.

\medskip
\textbf{Case 1.} First let $\III_j(\yy)=[y_j,y_{j+1}] \subseteq (0,1)$. Note that then $0<j<n$, as $y_0=0<y_j< y_{j+1} <1=y_{n+1}$ (the inequality $y_j<y_{j+1}$ has been clarified above).

We apply Lemma \ref{lem:widening} (c) with $\alpha:=y_j$, $a:=y_j+h/r_j$, $b:=y_{j+1}-h/r_{j+1}$, $\beta:=y_{j+1}$, $p:=r_j$ and $q:=r_{j+1}$, where $h>0$ is so small that $a<b$. We obtain for all $t\in A:=[0,1] \setminus \rint \III_j(\yy) = \cup_{i=0, i\ne j}^n \III_i(\yy)$ that
\begin{equation}\label{eq:maximin2}
r_j K(t-y_j) + r_{j+1} K(t-y_{j+1}) < r_j K(t-(y_j+\tfrac{h}{r_j})) + r_{j+1} K(t-(y_j+\tfrac{h}{r_{j+1}})).
\end{equation}
We define a new node system $\yy'$  by setting $y_j':=y_j+h/r_j$ and $y_{j+1}':=y_{j+1}-h/r_{j+1}$ and the rest of the nodes unchanged: $y'_i:=y_i$ for $i\ne j,\: j+1$. By the choice of $h>0$ we have $\yy'\in S$. By taking, if necessary, a smaller $h>0$, we can ensure   $m_j(\yy')>\mul(\yy)$ (continuity of $m_j$, Proposition \ref{prop:cont}).
Now adding $J(t) + \sum_{i\ne j, j+1} r_i K(t-y_i)$ to both sides of \eqref{eq:maximin2}, the left-hand side becomes $F(\yy,t)$, and the right-hand side becomes $F(\yy',t)$.

Let now $i\in\{0,1,\dots,n\}\setminus \{ j\}$ be any index and consider $m_i(\yy')$ and $m_i(\yy)>-\infty$ (recall $\yy$ is non-singular).
 The Trivial Lemma \ref{lem:trivi} with  $A:=\III_i(\yy) \subseteq B:=\III_i(\yy')$ yields $-\infty < \mul(\yy) \le m_i(\yy)=\max_A F(\yy,t) < \max_B F(\yy',\cdot)= m_i(\yy')$.
As we have already ensured $\mul(\yy)<m_j(\yy')$, we  find $\mul(\yy)<m_i(\yy')$ for all $i\in \{0,1,\dots,n\}$, whence we conclude $\mul(\yy)<\mul(\yy')$, a contradiction with the maximality of $\mul(\yy)$. We arrive at $m_j(\yy)=\mul(\yy)$ in this case.

\medskip
\textbf{Case 2.} Suppose $y_j=0$. As $\yy$ is a non-degenerate node system, we must have $j=0$ and $\III_j(\yy)=\III_0(\yy)=[0,y_1]$, $y_1>0$.
We will consider the new node system $\yy'$ with $y_1'=y_1-h$ and the rest unchanged: $y'_i=y_i$ for $i=2,\ldots,n$. With $0<h<y_1$ 
the new node system $\yy'$ also belongs to $S$. As above, for small enough $h>0$, 
continuity of $m_0$ (see Proposition \ref{prop:cont}) furnishes $m_0(\yy')>\mul(\yy)$.

Let $i\in\{1,2,\dots,n\}$ be arbitrary and consider $\III_i(\yy)$ and $\III_i(\yy')$. 
Obviously, $\III_i(\yy)\subseteq \III_i(\yy')$. Further, $-\infty<\mul(\yy)\le m_i(\yy)$. 
In view of strict monotonicity of $K$, we obviously have $r_1K(t-y_1) < r_1K(t-y_1+h)$ for every $t\in \III_i(\yy)$. 
Adding $J(t)+\sum_{i=2}^n  r_iK(t-y_i)$ to this inequality, the left-hand side becomes $F(\yy,t)$  and the right-hand side becomes $F(\yy',t)$. 
Applying the Trivial Lemma \ref{lem:trivi} with $A:=\III_i(\yy)$ and $B:=\III_i(\yy')$ 
we obtain $-\infty<\mul(\yy)\le m_i(\yy) <m_i(\yy')$ for each $i=1,2,\ldots,n$. 
In fact, also $-\infty<\mul(\yy)<m_0(\yy')$ was guaranteed above, which then furnishes  $\mul(\yy)<\min_i m_i(\yy')=\mul(\yy')$, contradicting the maximality of $\mul(\yy)$.
Therefore, $m_0(\yy)=\mul(\yy)$.

\medskip
\textbf{Case 3.} The case $y_{j+1}=1$ is completely analogous to Case 2.

\medskip\noindent Cases 1-3 altogether yield that $\yy$ is an equioscillation point, as claimed.
\end{proof}

\begin{corollary}\label{cor:minimaxmaximin-spec} Let $K$ be a singular \eqref{cond:infty}, strictly concave and (strictly) monotone \eqref{cond:smonotone} kernel function, and let $J$ be an upper semicontinuous field function.

Then $M(S)=m(S)$ and there exists a unique equioscillation point $\ww\in \oS$, which, in fact, belongs to $Y\subseteq S$. This point $\ww$ is the unique minimax point in $\oS$, i.e., $\mol(\ww)=M(S)$, and it is the unique maximin point in $\oS$, i.e., $\mul(\ww)=m(S)$.
In particular, the so-called Sandwich Property holds: for any node system $\xx \in S$ we have $\mul(\xx)\le M(S)=m(S)\le \mol(\xx)$.
\end{corollary}

\begin{proof}
The previous two theorems give that both minimax and maximin points must be equioscillation node systems.
Now, points in $\oS\setminus Y$ cannot be equioscillation
points, as degenerate or singular points $\xx$
satisfy $m_i(\xx)=-\infty$ for some $i\in \{0,1,\dots,n\}$
while $\mol(\xx)>-\infty$.
By Theorem \ref{thm:homeo3} the difference mapping $\Phi$ is a homeomorphism between $Y$ and $\RR^n$.
In particular, there exists exactly one pre-image
of $\textbf{0}$, i.e.,{}
only one equioscillation point in $Y$,
and hence also in $\oS$ in general.
As a result, both the maximin
and minimax points must coincide with this unique
equioscillation point (Theorems \ref{thm:EquiThm} and \ref{thm:maximinthm}).
\end{proof}

\begin{corollary}\label{cor:minimaxmaximin-infty} Let $K$ be a singular \eqref{cond:infty} and monotone \eqref{cond:monotone} kernel function, and let $J$ be an upper semicontinuous field function.

Then $M(S)=m(S)$ and there exists some node system $\ww\in \oS$, also belonging to $Y$, 
such that it is an equioscillation point and $\mul(\ww)=m(S)=M(S)=\mol(\ww)$.

In particular, the so-called Sandwich Property holds: for any node system $\xx \in S$ we have $\mul(\xx)\le M(S)=m(S)\le \mol(\xx)$, and $M(S)=m(S)$ is the unique equioscillation value.

If in addition the kernel $K$ satisfies \eqref{cond:smonotone}, then the point $\ww$ is the unique equioscillation point.
\end{corollary}
\begin{proof}
For the proof, we first apply the previous corollary in the situation with the same field function $J$ and the modified kernel functions $K^{(\eta)}(t):=K(t)+\eta \sqrt{|t|}$. If $\eta>0$, then $K^{(\eta)}$ is strictly concave and strictly monotone, thus Corollary \ref{cor:minimaxmaximin-spec} applies and provides node systems $\ww_\eta$ with the three asserted properties: $m^{(\eta)}(S)=M^{(\eta)}(S)=\mul^{(\eta)}(\ww_\eta)=\mol^{(\eta)}(\ww_\eta)$, where the notation refers to the corresponding quantities with the use of the kernel $K^{(\eta)}$. With a similar notation for the sum of translates function and putting $R:=\sum_{i=1}^n r_i$, it is obvious that $F(\xx,t) \le F^{(\eta)}(\xx,t) \le F(\xx,t)+\eta R$ for all $\xx \in \oS$ and $t\in [0,1]$. Therefore also $m_i(\xx)\le m_i^{(\eta)}(\xx)\le m_i(\xx)+\eta R$  and hence $m^{(\eta)}_i(\xx)\to m_i(\xx)$ for every $i=0,1,\dots, n$ and $\xx\in \oS$.
By compactness of $\oS$ we can take a convergent subsequence $(\ww_{1/{k_\ell}})$ of $(\ww_{1/k})$ with limit $\ww:=\lim_{\ell \to \infty} \ww_{1/k_\ell}$.  Moreover, by continuity of $m_i$ (see Proposition \ref{prop:cont}) we obtain
\begin{align*}
m_i(\ww)&=\lim_{\ell\to \infty} m_i(\ww_{1/k_\ell})\leq \liminf_{\ell\to \infty}m_i^{(1/k_\ell)}(\ww_{1/k_\ell})\leq \limsup_{\ell\to \infty}m_i^{(1/k_\ell)}(\ww_{1/k_\ell})\\
&\leq \limsup_{\ell\to \infty} \Bigl(m_i(\ww_{1/k_\ell})+R/k_\ell\Bigr)=m_i(\ww),
\end{align*}
that is  $\lim_{\ell\to\infty} m_i^{(1/k_\ell)}(\ww_{1/k_\ell}) =m_i(\ww)$. Therefore, $\ww$ is an equioscillation point in the case of the kernel $K$, and hence $\ww\in Y$.
 Let $\xx$ be a minimum point of $\mol$ on $\oS$. Then
 \[
M(S)=\mol(\xx)=\lim_{\eta\to 0} \mol^{(\eta)}(\xx)\geq\limsup_{\eta\to 0} M^{(\eta)}(S)\geq \liminf_{\eta\to 0} M^{(\eta)}(S)\geq M(S),
 \]
where the last inequality obviously follows from $K^{(\eta)}\geq K$.
Therefore, $M(S)=\lim_{\eta\to 0} M^{(\eta)}(S)$, whence we can also conclude
\[
M(S)=\lim_{\ell\to \infty} M^{(1/k_\ell)}(S)=\lim_{\ell\to \infty} \mol^{(1/k_\ell)}(\ww_{1/k_\ell})=\mol(\ww),
\]
i.e., $\ww$ is a minimum point of $\mol$. Since for the $\eta$-perturbed kernel we have $M^{(\eta)}(S)=m^{(\eta)}(S)$, we infer that
\[\lim_{\eta\to 0}\mul^{(\eta)}(\ww_\eta)=\lim_{\eta\to 0}
m^{(\eta)}(S)=\lim_{\eta\to 0} M^{(\eta)}(S)\]
 exists, and equals $M(S)$.

On the other hand, $m(S)\leq m^{(\eta)}(S)=\mul^{(\eta)}(\ww_\eta)$, so
\begin{align*}
m(S)\leq \lim_{\eta\to 0}  m^{(\eta)}(S)=\lim_{\eta\to 0}\mul^{(\eta)}(\ww_\eta) =\lim_{\ell\to \infty}\mul^{(1/k_\ell)}(\ww_{1/k_\ell}) =\mul(\ww)\leq m(S).
\end{align*}
Hence $\lim_{\eta\to 0}  m^{(\eta)}(S)= m(S)$ and $\mul(\ww)=m(S)$, i.e., $\ww$ is a maximum point of $\mul$.

Finally, uniqueness of the equioscillation point under condition \eqref{cond:smonotone} follows from Theorem \ref{thm:homeo3}.
\end{proof}

\section{Intertwining}
\label{sec:inter}

The main result of this section,  Theorem
\ref{thm:Intertwining}, shows that for different node systems $\xx,\yy\in Y$ it is not possible to have $m_j(\xx)\leq m_j(\yy)$ for every $j\in \{0,1,\dots,n\}$, i.e., majorization cannot occur (cf.{} \cite{TLMS2018} for the terminology). 
In other words, for two different node systems $\xx,\yy\in Y$ both $m_j(\xx)<m_j(\yy)$ and $m_i(\xx)>m_i(\yy)$ holds for some $i,j\in \{0,1,\dots,n\}$, a property that is natural to be called \emph{intertwining}.
For the proof, we need
the following
perturbation type lemma
which is interesting on
its own.

\begin{lemma}[\textbf{General maximum perturbation lemma}]\label{lem:gen-max-perturbation}
Let $n\in\NN$ be a natural number, let $r_1,\dots, r_n>0$, let $J:[0,1]\to \uR$ be an upper semicontinuous $n$-field function, and
let $K$ be a kernel function satisfying the monotonicity condition \eqref{cond:monotone}.
Consider the sum of translates function $F$ as in \eqref{eq:Fsum}.

Let $\ww \in S$ be a non-degenerate node system,
and let $\cI \cup \cJ =\{0,1,\ldots,n\}$ be a non-trivial partition.
Then there exists $\ww'\in S\setminus \{\ww\}$ arbitrarily close to $\ww$ with
\begin{align}
\label{eq:genpertlemma-function-I}
&F(\ww',t)\le F(\ww,t) \text{ for all } t \in \III_i(\ww') \quad \text{and} \quad \III_i(\ww')\subseteq \III_i(\ww) \quad \text{for all}\  i\in\cI;\\
\label{eq:genpertlemma-function-J}
& F(\ww',t)\ge F(\ww,t)  \text{ for all } t \in \III_j(\ww)  \quad \text{and} \quad \III_j(\ww')\supseteq \III_j(\ww) \quad \text{for all}\  j\in\cJ.
\end{align}
As a result, we also have
\begin{equation}\label{eq:genpertlemma-max}
m_i(\ww')\le m_i(\ww) \text{ for }i\in\cI\quad \text{and} \quad   m_j(\ww')\ge m_j(\ww) \text{ for }j\in\cJ
\end{equation}
for the corresponding interval maxima.

Moreover, if $K$ is strictly concave (and hence by condition \eqref{cond:monotone} also strictly monotone), then the inequalities in \eqref{eq:genpertlemma-function-I} and \eqref{eq:genpertlemma-function-J} are strict for all points in the respective intervals where $J(t)\ne - \infty$.

Furthermore, the inequalities in \eqref{eq:genpertlemma-max} are also strict for all indices $k$ with non-singular $\III_k(\ww)$; in particular, for all indices if $\ww\in Y$.
\end{lemma}

\begin{proof} Before the main argument, we observe that the assertion in \eqref{eq:genpertlemma-max} is indeed a trivial consequence of the previous inequalities \eqref{eq:genpertlemma-function-J} and \eqref{eq:genpertlemma-function-I}, so we need not give a separate proof for that.

A second important observation is the following. With the pure sum of translates function $f$ we write $F(\ww,t)=f(\ww,t)+J(t)$, and so the inequalities \eqref{eq:genpertlemma-function-I} and \eqref{eq:genpertlemma-function-J} follow from
\begin{align}
	\label{eq:genpertlemma-fI}
	&f(\ww',t)\le f(\ww,t) \ (\forall t \in \III_i(\ww')) \quad \text{and} \quad \III_i(\ww')\subseteq \III_i(\ww) \quad \text{for all}\  i\in\cI;\\
	\label{eq:genpertlemma-fJ}
& f(\ww',t)\ge f(\ww,t) \ (\forall t \in \III_j(\ww)) \quad \text{and} \quad \III_j(\ww')\supseteq \III_j(\ww) \quad \text{for all}\  j\in\cJ.
\end{align}
Moreover, strict inequalities for all points $t$ with $J(t)\ne-\infty$ will follow in \eqref{eq:genpertlemma-function-I} and \eqref{eq:genpertlemma-function-J} if we can prove strict inequalities in
\eqref{eq:genpertlemma-fI} and \eqref{eq:genpertlemma-fJ}
\emph{for all values of $t$} in the said intervals.

Furthermore, in case we have strict inequalities in \eqref{eq:genpertlemma-fI} and \eqref{eq:genpertlemma-fJ} for all points $t$, then for non-singular $\III_k(\ww)$ this entails strict inequalities also in \eqref{eq:genpertlemma-max} (for the corresponding $k$; and for all $k$ if $\ww\in Y$). To see this, one may refer back to the Trivial Lemma \ref{lem:trivi} with $\{f,g\}=\{f(\ww,\cdot),f(\ww',\cdot)\}$, $h=J$, $\{A,B\}=\{\III_k(\ww),\III_k(\ww')\}$.

In the next, main part of the argument we prove \eqref{eq:genpertlemma-function-I}, \eqref{eq:genpertlemma-function-J},  \eqref{eq:genpertlemma-fI}, and \eqref{eq:genpertlemma-fJ} by induction on $n$ for any $n$-field function and any kernel function.

If $n=1$ and $\cI=\{0\}$, $\cJ=\{1\}$, then $\ww'=(w_1')=(w_1+h)$ and if $\cJ=\{0\}$,  $\cI=\{1\}$, then $\ww'=(w_1')=(w_1-h)$ works with any $0<h<\min(w_1,1-w_1)$.  For this  only monotonicity resp.{} strict monotonicity of the kernel is needed, and hence \eqref{eq:genpertlemma-fI}, \eqref{eq:genpertlemma-fJ} follow readily, while \eqref{eq:genpertlemma-function-I}, \eqref{eq:genpertlemma-function-J} follow by the preliminary observations made above.

Let now $n>1$ and assume, as inductive hypothesis, the validity of the assertions 
for $\widetilde{n}:=n-1$ for any choice of kernel and $\widetilde{n}$-field functions.

\textbf{Case 1}. If some of the partition sets $\cI, \cJ$ contain neighboring indices $k, k+1$, then we consider the kernel function $\widetilde{K}:=K$, and the $\widetilde{n}$-field function  $\widetilde{J}:=K(\cdot-w_k)$ with now the sum of translates function $\widetilde F$ formed by using $\widetilde{n}=n-1$ translates with respect to the node system
\[
\widetilde{\ww}:=(w_1,w_2,\ldots,w_{k-1},w_{k+1},\ldots,w_n).
\]
Formally, the indices change: $\widetilde{w}_\ell=w_\ell$ for $\ell=1,\ldots,k-1$, but $\widetilde{w}_\ell=w_{\ell+1}$ for $\ell=k,\ldots,n$, the $k$th coordinate being left out.

We apply the same change of indices in the partition: $k$ is dropped out (but the corresponding index set $\cI$ or $\cJ$ will not become empty, for it contains $k+1$); and then shift indices one left for $\ell>k$: so
\begin{align*}
	\widetilde{\cI}&:=\{i \in \cI: i<k\} \cup \{i-1 \in \cI: i>k\}
\intertext{and}
\widetilde{\cJ}&:=\{j \in \cJ: j<k\} \cup \{j-1 \in \cJ: j>k\}.
\end{align*}
Observe that $\widetilde{F}(\widetilde{\ww},t)=f(\ww,t)$  for all $t\in [0,1]$, while
\[
\III_\ell(\widetilde{\ww})=\begin{cases}
\III_\ell(\ww) \quad & \text{if} \quad \ell <k, \\
\III_k(\ww)\cup \III_{k+1}(\ww) \quad & \text{if} \quad \ell =k, \\
\III_{\ell+1}(\ww) \quad & \text{if} \quad \ell >k.
\end{cases}
\]
If $\ww'$ is close enough to $\ww$, then a similar correspondence holds for  $\ww' \in S^{(n)}$ and $\widetilde{\ww}'\in S^{(\widetilde{n})}$  (where $S^{(n)}$ and $S^{(\widetilde{n})}$ denote the simplices of the corresponding dimension). 
We will use this only with $w'_k=w_k$ remaining the same. 
In this case using that $k$ and $k+1$ belong to the same index set $\cI$ or $\cJ$, 
it is easy to check that $\III_i(\widetilde{\ww}') \subseteq \III_i(\widetilde{\ww})$  
for all $i\in \widetilde{\cI}$ is equivalent to $\III_i(\ww')\subseteq \III_i(\ww)$ for all $i\in {\cI}$,
and $\III_j(\widetilde{\ww}') \supseteq \III_j(\widetilde{\ww})$  
for all $j\in \widetilde{\cJ}$ is equivalent to $\III_j(\ww')\supseteq \III_j(\ww)$ for all $j\in {\cJ}$.
Therefore an application of the inductive hypothesis yields the assertions \eqref{eq:genpertlemma-fI}, \eqref{eq:genpertlemma-fJ} in this case.
Whence, by the preliminary observations also \eqref{eq:genpertlemma-function-I}, \eqref{eq:genpertlemma-function-J} follow.

\textbf{Case 2}. It remains to prove the assertion when $\cI, \cJ$ contain no neighboring indices: 
so  $\cI$ and $\cJ$ partition 
$\{0,1,\ldots,n\}$ into the subsets of odd and even natural numbers up to $n$.
We can assume that $\cI=(2\NN_0+1) \cap \{0,1,\ldots,n\}$ and $\cJ=2\NN_0 \cap \{0,1,\ldots,n\}$ 
(the other case can be handled analogously).

We emphasize here that it is important that $\ww \in S$ is non-degenerate.
This allows, for sufficiently small $\delta>0$, to move any $w_\ell$ within a
distance $\delta>0$ still keeping that the perturbed node system $\ww'$ belongs to $S$.
We fix such a $\delta>0$ at the outset and consider perturbations $\ww'$ of $\ww$ only within distance $\delta$ from now on.
Our new perturbed node system $\ww'$ will be, with an
arbitrary $0<h<\delta/\max\{r_1,\dots, r_n\}$, the system
\begin{equation}
\label{eq:wprimeoddeven}
\ww':=(w_1',\ldots,w_n') \quad \text{with} \quad w'_\ell:=w_\ell-(-1)^\ell \frac{1}{r_\ell} h, \quad \ell=1,2,\ldots,n.
\end{equation}
Obviously, $\III_j(\ww') \supseteq \III_j(\ww)$ holds for all $j\in \cJ$, and $\III_i(\ww') \subseteq \III_i(\ww)$ for all $i\in \cI$.

Take now an even indexed interval $\III_{2k}(\ww)=[w_{2k},w_{2k+1}]$, so $2k \in \cJ$.
Our change of nodes can now be grouped as \emph{pairs of changing nodes} $w_{2\ell-1},w_{2\ell}$
among $w_1,\dots w_{2k}$, and then again among $w_{2k+1},\dots, w_{2\lfloor n/2\rfloor}$,
plus a left-over change of $w_n$ in case $n$ is odd.
Now, \emph{the pairs} are always changed so that the intervals in between shrink, and shrink exactly as is described in Lemma \ref{lem:widening}.
We apply this lemma for each pair of such nodes with the choices $a=w'_{2\ell-1}$, $b=w'_{2\ell}$, $\alpha=w_{2\ell-1}$, $\beta=w_{2\ell}$, $p=r_{2\ell-1}$, $q=r_{2\ell}$.
This gives that for each such pair of changes, for $t$ \emph{outside of the enclosed interval} $(w_{2\ell-1},w_{2\ell})$ we have
\begin{equation}\label{eq:genpert1}
	r_{2\ell-1}K(t-w'_{2\ell-1})+r_{2\ell}K(t-w'_{2\ell})\geq r_{2\ell-1}K(t-w_{2\ell-1})+r_{2\ell}K(t-w_{2\ell}).
\end{equation}
Note that $\III_{2k}(\ww)$, hence any $t \in \III_{2k}(\ww)$, is \emph{always outside of the intervals}, therefore \eqref{eq:genpert1} holds.
If there is a left-over, unpaired change, then $n$ is odd, the respective node $w_n$ is increased, and \emph{now by monotonicity} we conclude for $t\in \III_{2k}(\ww)$  that $K(t-w_n')=K(t-w_n-h/r_n)\geq K(t-w_n)$.
Altogether, we find with $\eta:=1$ for $n$ odd and $\eta:=0$ for $n$ even that
\begin{align}\label{eq:groupedchange} \notag
f(\ww,t)& = \sum_{\ell=1}^{\lfloor n/2\rfloor } \left(r_{2\ell-1}K(t-w_{2\ell-1})+r_{2\ell}K(t-w_{2\ell})\right) + \eta K(t-w_n)
\\& \le \sum_{\ell=1}^{\lfloor n/2\rfloor } \left(r_{2\ell-1}K(t-w_{2\ell-1}')+r_{2\ell}K(t-w_{2\ell}')\right) + \eta K(t-w_n')
 =f(\ww',t).
\end{align}
Furthermore, all the appearing inequalities are strict in case $K$ is strictly monotone (and hence is strictly concave).
We have proved \eqref{eq:genpertlemma-fJ}, even with strict inequality under appropriate assumptions.

The proof of \eqref{eq:genpertlemma-fI}  runs analogously
by grouping the change of nodes as a change of a singleton $w_1$,
and  then  of pairs $w_{2\ell}, w_{2\ell+1}$ for $\ell=1,\dots,\lfloor (n-1)/2\rfloor$,
and of another singleton $w_n$ if $n$ is even.
\end{proof}

\begin{theorem}[\textbf{Intertwining theorem}]\label{thm:Intertwining} 
Let $n\in\NN$, let $r_1,\dots, r_n>0$, let $K$ be a singular \eqref{cond:infty}, 
strictly concave and (strictly) monotone \eqref{cond:smonotone} kernel function 
and let $J:[0,1]\to \uR$ be an upper semicontinuous $n$-field function.

Then for nodes $\xx,\yy\in Y$ majorization cannot hold, i.e., the coordinatewise inequality $\mv(\xx)\le \mv(\yy)$ can only hold if $\xx=\yy$.
\end{theorem}

\begin{proof}
Take two node systems $\xx, \yy \in Y$ and assume that majorization holds between them: say $\mv(\xx)\le \mv(\yy)$ in the sense that $m_i(\xx)\le m_i(\yy)$ for $i=0,1,\ldots,n$. We need to show that in fact $\xx=\yy$.

First, if $\mv(\xx)=\mv(\yy)$, then of course $\Phi(\xx)=\Phi(\yy)$, hence in view of the Homeomorphism Theorem \ref{thm:homeo3} (which requires condition \eqref{cond:infty})
 only $\xx=\yy$ is possible.

So assume that $\mv(\xx)\ne\mv(\yy)$, and so there exists $i$ with $m_i(\xx)<m_i(\yy)$.
Let us introduce the following two (``maximal'' and ``minimal'') distance functions
\begin{align*}
d(\zz,\yy)&:=\max_{i=0,1,\ldots,n} (m_i(\yy)-m_i(\zz)),
\\ \rho(\zz,\yy)&:=\min_{i=0,1,\ldots,n} (m_i(\yy)-m_i(\zz)).
\end{align*}
Then $\rho(\zz,\yy)\leq d(\zz,\yy)$. Moreover, $0\leq \rho(\zz,\yy)$ if and only if $\mv(\zz)\leq \mv(\yy)$, and for $d_0:=d(\xx,\yy)$ we have $d_0>0$.
Consider the set
\[
  Z:=\bigl\{\zz\in \oS\ :\  \mv(\zz)\le \mv(\yy), d(\zz,\yy)\le d_0\bigr\}.
\]
Obviously, $\xx \in Z$, hence $Z \ne \emptyset$. By Proposition \ref{prop:cont} the distance functions $d(\cdot,\yy), \rho(\cdot,\yy)$ are (extended) continuous on $\oS$ and continuous on $Y$. In fact,  $\zz\in Z$ cannot be singular, thus $Z\subseteq Y$, where the distance functions are  continuous. Further, as the intersection of $\le$ level sets of continuous functions (see Proposition \ref{prop:cont}) $Z$ is closed, and therefore compact.

Here we arrive at the key of our argument: We now maximize $\rho(\cdot,\yy)$ on the compact set $Z$. Surely, $\rho(\cdot,\yy)$ can be at most $d_0$ on $Z$.
Let $\zz_0\in Z$ 
be a maximum point of $\rho(\cdot,\yy)$, and set  $\rho_0:=\rho(\zz_0,\yy)\le d_0$. 
We  claim that the difference $m_i(\yy)-m_i(\zz_0)$ is constant $\rho_0$ for all $i=0,1,\ldots,n$.

Indeed, if this is not the case, then by means of Lemma \ref{lem:gen-max-perturbation}, we can perturb $\zz_0$ to another node system $\ww$ with a larger $\rho$ value. In detail: assume for a contradiction that $\mv(\yy)-\mv(\zz_0)\ne \rho_0 \textbf{1}$. Note that then we also have $\rho_0<d_0$, for in case $\rho_0=d_0$ we must have $\mv(\zz_0)=\mv(\yy)-d_0\textbf{1}=\mv(\yy)-\rho_0\textbf{1}$, contradicting the assumption.

Now let us define the index sets
\begin{align*}
&\cI:=\{i \in \{0,1,\dots,n\} \ :\   m_i(\zz_0)=m_i(\yy)-\rho_0\},
\\ &\cJ:=\{j \in \{0,1,\dots,n\} \ :\  m_j(\yy)-d_0 \le m_j(\zz_0)< m_j(\yy)-\rho_0 \}.
\end{align*}
As $m_i(\yy)-m_i(\zz_0)$ is not constant in $i$, we certainly have indices in both index sets $\cI, \cJ$. Moreover, in view of $\zz_0 \in Z$ we have $m_k(\zz_0)\in [m_k(\yy)-d_0,m_k(\yy)-\rho_0]$ for all $k=0,1,\ldots,n$. Therefore, $\cI\cup\cJ$ is in fact a non-trivial partition of $\{0,1,\dots,n\}$. Thus,  Lemma \ref{lem:gen-max-perturbation} applies for these indices and to the non-singular, non-degenerate point $\zz_0$, resulting in another node system $\ww\in Y\setminus \{\zz_0\}$ arbitrarily close to $\zz_0$ with $m_i(\ww)<m_i(\zz_0)$ for $i\in \cI$ and $m_j(\ww)>m_j(\zz_0)$ for $j\in\cJ$. Since for $j\in \cJ$ we have $m_j(\zz_0)<m_j(\yy)-\rho_0$, by the continuity of the functions $m_i$ (Proposition \ref{prop:cont}) if $\ww$ is sufficiently near to $\zz_0$ we have
that $m_j(\ww)<m_j(\yy)-\rho_0 $ for all $j\in\cJ$. Of course, for these indices $j\in \cJ$ also the inequality $m_j(\ww)\ge m_j(\yy)-d_0$ remains valid, since $m_j(\ww)> m_j(\zz_0) \ge m_j(\yy)-d_0$ for $j\in\cJ$.

Similarly, after perturbation we find $m_i(\ww)<m_i(\zz_0)=m_i(\yy)-\rho_0$ for all $i\in\cI$, and, by continuity, in a sufficiently small neighborhood of $\zz_0$ also the inequality $m_i(\ww)\ge m_i(\yy)-d_0$  holds. (Here of course we need that $0<d_0-\rho_0$, and use continuity.)

Altogether we find $\ww \in Z$, but $m_k(\ww)<m_k(\yy)-\rho_0$ for all $k=0,1,\ldots,n$, whence $\rho(\ww,\yy)>\rho_0$ follows, a contradiction with the choice of $\zz_0$ as maximizing $\rho(\cdot,\yy)$ on $Z$. This proves that $\zz_0 \in Z$ can only be a point with coordinates of $\mv(\zz_0)$ having constant distance $\rho_0$ from the respective coordinates of $\mv(\yy)$: $m_k(\zz_0)=m_k(\yy)-\rho_0$, for $k=0,1,\ldots,n$.

It follows that $\yy,
\zz_0$ are two points
of $Y$ with equal
difference vectors:
$\Phi(\yy)=\Phi(\zz_0)$.
By Theorem \ref{thm:homeo3} $\Phi$ is, in particular, injective, hence $\zz_0=\yy$.
It follows that $\rho_0=0$, and the maximum of $\rho$-distances
between points of $Z$ to $\yy$---and
therefore $\rho$-distances of \emph{any node system} $\zz \in Z$
from $\yy$---can only be zero (since $\rho(\cdot,\yy)\geq 0 $ on $Z$).
That is, all $\zz\in Z$ are maximum points for the $\rho$-distance: $\rho(\zz,\yy)=\rho_0=0$ for all $\zz\in Z$. It follows that for any $\zz \in Z$ the same applies as for the selected $\zz_0$ and we conclude that $Z=\{\yy\}$. Since $\xx\in Z$, it follows that $\xx=\yy$, and that was to be proved.
\end{proof}

\begin{remark}
Similar non-majorization results are rare, but we may
compare to, e.g., Theorem 1, on p.~17 of \cite{Parthasarathy}.
If the kernel function $K$ and also the external field function $J$ are smooth,
then we may consider the Jacobi matrix
of $\Phi$ (the interval maxima difference function)
and the proof of the Homeomorphism Theorem \ref{thm:homeo3}---i.e., the proof of Theorem 2.1 in \cite{Homeo}---where it is shown
that this Jacobi matrix is diagonally dominant.
It is known that diagonally dominant matrices are so-called ``P-matrices''
(for more, we refer to pp.~134-137 of \cite{BermanPlemmons}), hence the condition of the cited Theorem 1 of \cite{Parthasarathy} is satisfied.
However, the conclusion of that result is far weaker than ours:
it excludes majorization only in case the nodes $\xx, \yy$ are ordered similarly coordinatewise: $x_i\le y_i \ (i=1,\ldots,n)$.
The generality that we get non-majorization for \emph{all node systems} $\xx, \yy \in Y$ can be attributed to the special setup,
where $\Phi$ is formed from differences of interval maxima of sum of translates functions satisfying our assumptions.
\end{remark}

\begin{corollary}
\label{cor:Chebyshev}
Consider the (almost) two centuries
old classical Chebyshev problem, where in
our terminology $K(t):=\log|t|$, $J(t)\equiv 0$,
and so we have strict concavity and monotonicity.
Then for any two node systems
$\xx, \yy \in S$
we necessarily have some indices
$0\le i\ne j \le n$ such that
\begin{align*}\label{eq:ChebyshevIntertwining}
& \max_{t\in \III_i(\xx)} \Bigl|\prod_{k=1}^n (t-x_k)\Bigr|
<
\max_{t\in \III_i(\yy)}
\Bigl|\prod_{k=1}^n (t-
y_k)
\Bigr|,
\\& \max_{t\in \III_j(\xx)}\Bigl |\prod_{k=1}^n (t-x_k)\Bigr|
>
\max_{t\in \III_j(\yy)}\Bigl |\prod_{k=1}^n
(t- y_k)
\Bigr|.
\end{align*}
\end{corollary}

\begin{figure}[h]
\begin{center}
\includegraphics[keepaspectratio,width=0.6\textwidth]{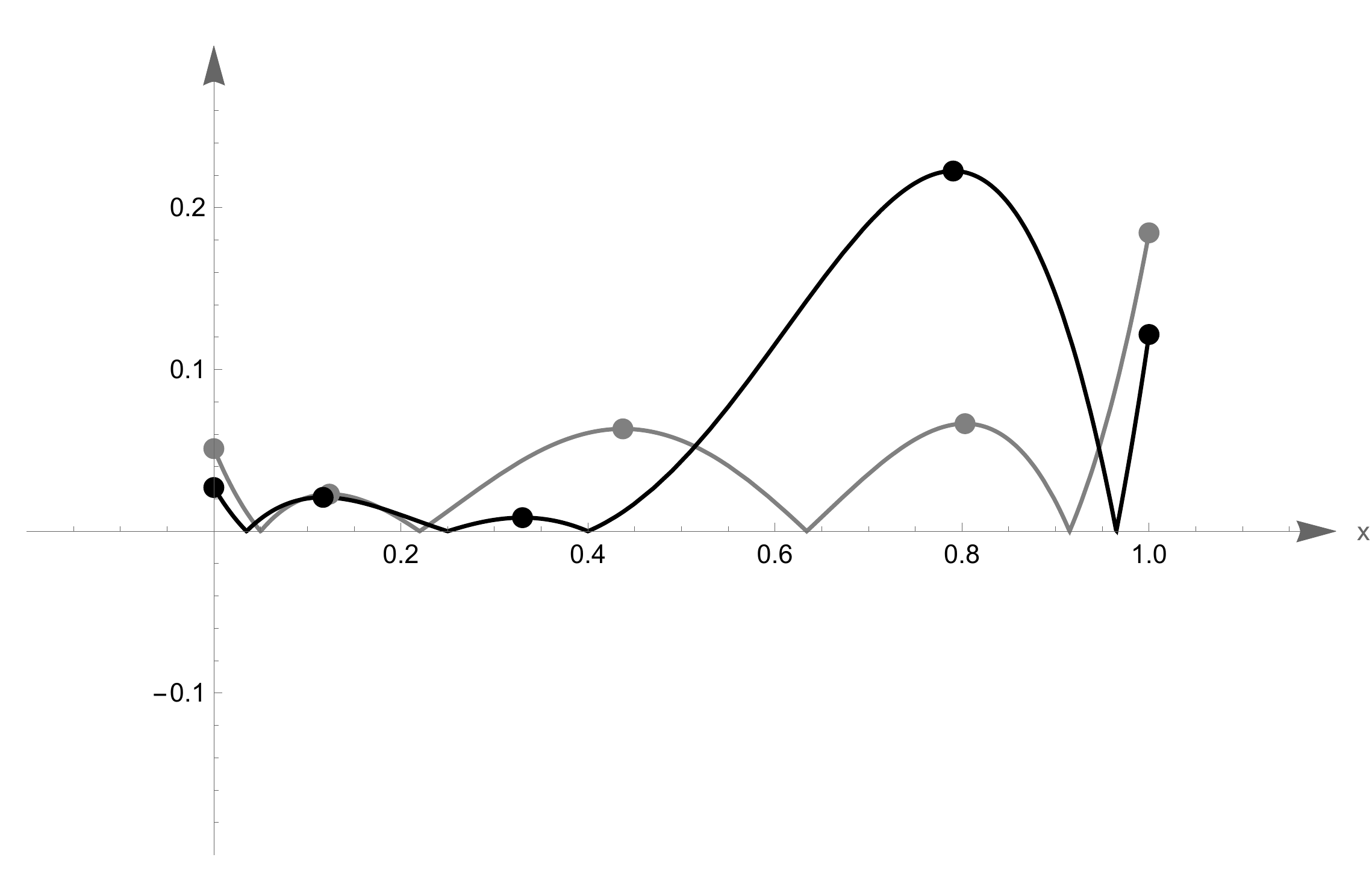}
\caption{Graphs of
$|(x-0.915)(x-0.634)(x-0.22)(x-0.05)|$ and
$|(x-0.965)(x-0.4)(x-0.25)(x-0.035)|$
 with dots at local maxima over $[0,1]$
in grey and black, respectively, 
cf. Corollary \ref{cor:Chebyshev}.}
\end{center}
\end{figure}

\begin{remark}
It seems that even in this classical situation the above general statement has not been observed thus far.
The special case when one of the node systems say $\xx$ is the extremal (equioscillating) node system $\ww$,
is well known and seems to be folklore.
However, comparison of arbitrary two node systems looks more complicated
and nothing was written about it in the literature what we could page through.
\end{remark}

\begin{corollary}[\textbf{Non-majorization theorem}]\label{cor:non-majorization0} 
Let $K$ be a singular \eqref{cond:infty} and  monotone \eqref{cond:monotone} kernel function, 
and let $J$ be an upper semicontinuous field function.
Then \emph{strict majorization} $m_i(\xx)>m_i(\yy)$ for every $i=0,1,\ldots,n$ cannot hold between any two node systems $\xx,\yy \in Y$.
\end{corollary}
\begin{proof} As in the proof of Corollary \ref{cor:minimaxmaximin-infty}, consider the modified kernel functions $K^{(\eta)}(t):=K(t)+\eta \sqrt{|t|}$, which are strictly concave and strictly monotone kernel functions. Let $\xx,\yy\in \oS$.
Then, as in the mentioned proof, for $\eta\downarrow 0$ we have  $\mv^{(\eta)}(\xx)\to \mv(\xx)$ and  $\mv^{(\eta)}(\yy)\to \mv(\yy)$. This implies that once $\mv(\xx) > \mv(\yy)$, we must have $\mv^{(\eta)}(\xx) > \mv^{(\eta)}(\yy)$ for every sufficiently small $\eta>0$, which is impossible by Theorem \ref{thm:Intertwining}, given that by condition $\xx \ne \yy$. Whence we conclude that $\mv(\xx) > \mv(\yy)$ is for no pair 
$\xx,\yy\in Y$
 possible.
\end{proof}

\begin{corollary}
Let $K$ be a singular \eqref{cond:infty} and
monotone \eqref{cond:monotone} kernel function and
let $J$ be an upper semicontinuous field function.
Let $\xx,\yy\in Y$ with
$\mul(\xx)=m(S)=M(S)= \mol(\yy)$.

Then there is a $j\in \{0,1,\dots,n\}$ such that
 \[
 m_j(\xx)=m(S)=M(S)=m_j(\yy).
 \]
\end{corollary}

\section{ A discussion of conditions}
\label{sec:necessity}
In this section we show by examples
that dropping conditions from our results
entail that the conclusions
may not hold true any more.
These justify using the given conditions,
even if at first glance assuming, e.g., monotonicity or strict concavity
may not seem to be natural
or necessary.
The examples here also highlight the generality of our statements,
where further, e.g.,
smoothness conditions were not supposed.
In our results only conditions which are shown here 
not to be simply dispensable, were assumed throughout. 
Below we talk about ``necessity'' of conditions in this, logically weaker sense of indispensability.

\smallskip

\begin{example}[\bf Necessity of singularity]

Let $n=2$, $J(t):=8 \sqrt{1-t} $ and
$K(t):=\sqrt{t+4}$ if $t\in [0,1]$
and $K(t):=K(-t)$ if $t\in[-1,0)$.
Then,  $J$ is a concave field function, $J\in C^\infty([0,1])$,
and
$K\in C^{\infty}([-1,1]\setminus\{0\})$ is a strictly concave kernel function,
further, $K$ is monotone as in \eqref{cond:monotone}
and $J$ and $K$
do not satisfy the condition \eqref{cond:infty}.

We have $M(S)=m(S)=-4$,
in $\overline{S}$ there are  unique equioscillation,  unique minimax
and unique maximin node systems and
all these are $(0,0)\in\partial S$.
In other words, almost all conclusions of the above theorems hold true,
except that this point of extrema and equioscillation is not in $S$,
but on the boundary $\partial S$.
\end{example}

The key observation is that
$\frac{d}{dt}F(\yy,t) <0$ for any  $\yy\in S$ at every $t\in [0,1]$.
Hence
\begin{align*}
m_0(\yy)=\ & \max\{F(\yy,t):\ 0\le t\le y_1\}
=F(\yy,0)=8+\sqrt{4+y_1}+
\sqrt{4+y_2},
\\
m_1(\yy)=\ & F(\yy,y_1)
=8\sqrt{1-y_1}+2+\sqrt{4+y_2-y_1},
\\
m_2(\yy)=\ & F(\yy,y_2)
=8\sqrt{1-y_2}+\sqrt{4+y_2-y_1}+2.
\end{align*}
By the observation,
$m_0(\yy)\ge m_1(\yy)$
with equality if and only if
$y_1=0$
and similarly,
$m_1(\yy)\ge m_2(\yy)$
with equality if and only if
$y_1=y_2$.
Also, $\mol(\yy)=m_0(\yy)$
and $\mul(\yy)=m_2(\yy)$.
Obviously, $m_0(\yy)$,
$y\in\overline{S}$ is minimal
if and only if $\yy=(0,0)$,
and $m_2(\yy)$ is maximal
if and only if $\yy=(0,0)$.
Therefore we obtain that $M(\overline{S})=m(\overline{S})=-4$
and these are attained at $\yy=(0,0)$ only
and there is a unique equioscillation configuration
in $\overline{S}$, namely $\yy=(0,0)$.

\medskip

For convenience, we introduce the kernel function $L_a(t):= \min(0,\log|t/a|)$.

\begin{example}[{\bf Necessity of monotonicity}]
Let $n=1$, $J(t):=\sqrt{t}$
and $K(t):=L_{0.1}(t) + 1-2t^2$.
Note that $J$ is a strictly concave field function
and $K$ is a  strictly concave kernel function and $K$ is singular,
but it does not satisfy any of the monotonicity conditions \eqref{cond:monotone} and \eqref{cond:smonotone}.

Then the global minimum $M(S)$ of $\mol$ is attained only at $\yy=(y_1)=(0)$,
so $\III_0$ is degenerate and $\yy\in \partial S$. Also, $m_0(\yy)=-\infty$ and $m_1(\yy)=F(0,1/4)=11/8$. Obviously, $F(\yy,\cdot)$ does not equioscillate.
\end{example}

Indeed, if  $0<y_{1}\le 1/2$,
then $\mol(\yy)\ge F(y_{1},y_{1}+1/4)=\sqrt{y_{1}+ 1/4}+7/8>1/2+7/8=11/8$
and if $1/2<  y_{1}\le 1$ then $\mol(\yy)\ge F(y_{1},y_{1}-1/4)=\sqrt{y_{1}-1/4} + 7/8 >11/8$.
So
$\mol(\yy)=11/8$ is
attained at $\yy=(y_{1})=(0)$ only.

\begin{example}[\bf Necessity of strict monotonicity and concavity]\label{examp:2}
Let $0<a<\frac{e}{1+e}$ be arbitrary and  $K(t):=L_a(t)$.
Set $n=1$. Let $b\in(0,1)$ satisfy $(a<)1-a/e<b<1$ and let $J$ be the characteristic function of the interval $[b,1]$.
Then
\begin{enumerate}[label={(\alph*)}]
\item $\mv(\xx)=(0,0)$ if and only if
$\xx=(x)$,
$x=1-a/e$, and this is the unique equioscillation point of $F$, moreover, $m(S)=M(S)=0$;
\item however, $\mul(\xx)=0$ for all $a\le x \le 1-a/e$ and thus $\mul(\xx)=m(S)=0$ is attained not only for the equioscillating node, but for several other, non-equioscillating ones.
\item In particular, majorization occurs on $Y$.
\end{enumerate}
Note that $K$ is a concave and monotone kernel function,
but not strictly concave or strictly monotone.

\begin{figure}[h]
\begin{center}
\includegraphics[keepaspectratio,width=0.4\textwidth]{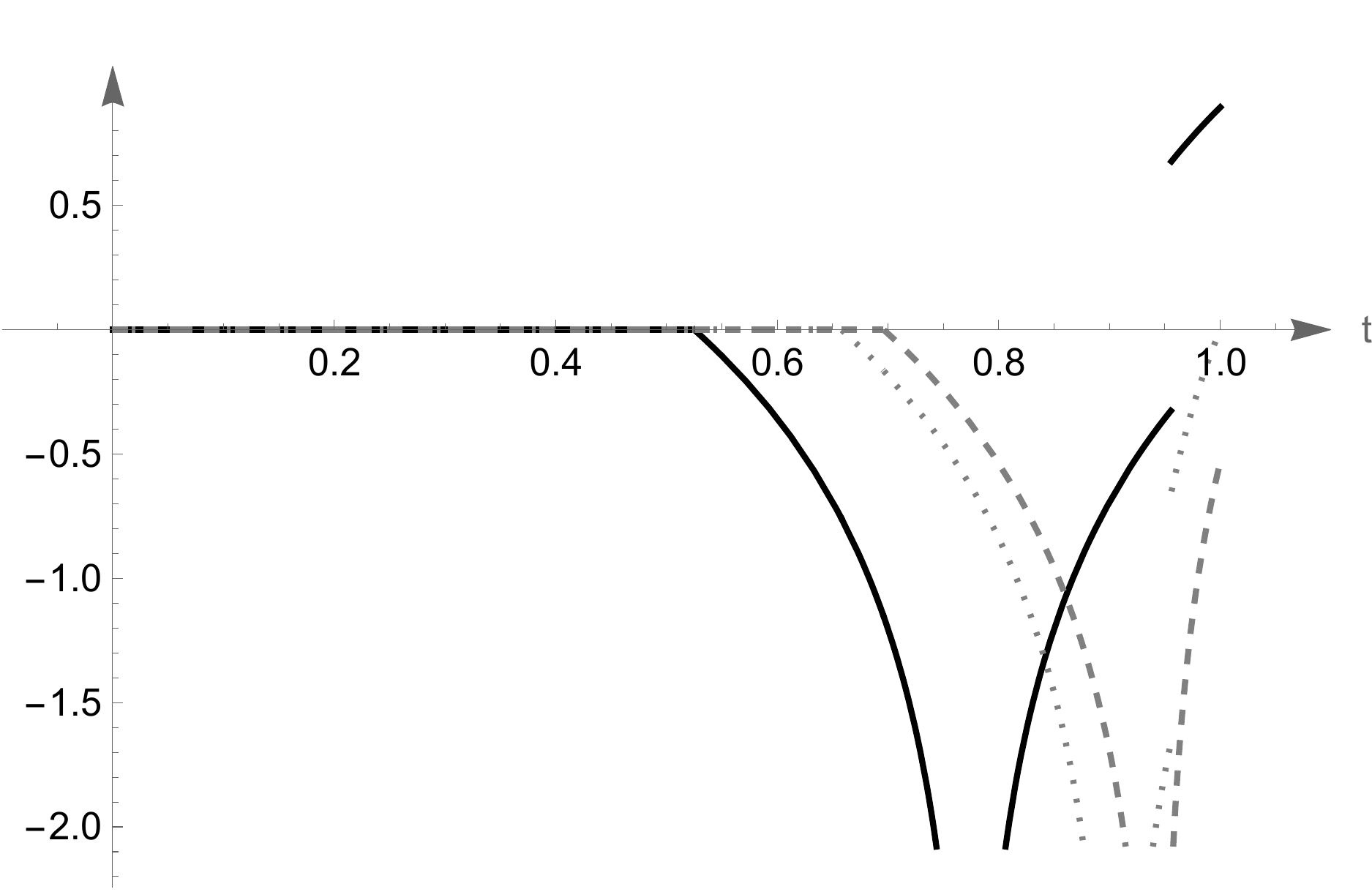}
\caption{The graphs of 
sum of translates functions $J(t)+K(t-x)$ when $a=1/4$, $b=0.955671$ and $x=0.775$ (black), $x=0.907$ (grey, dotted) and $x=0.946$ (grey, dashed), 
see Example \ref{examp:2}.
}
\
\includegraphics[keepaspectratio,width=0.4\textwidth]{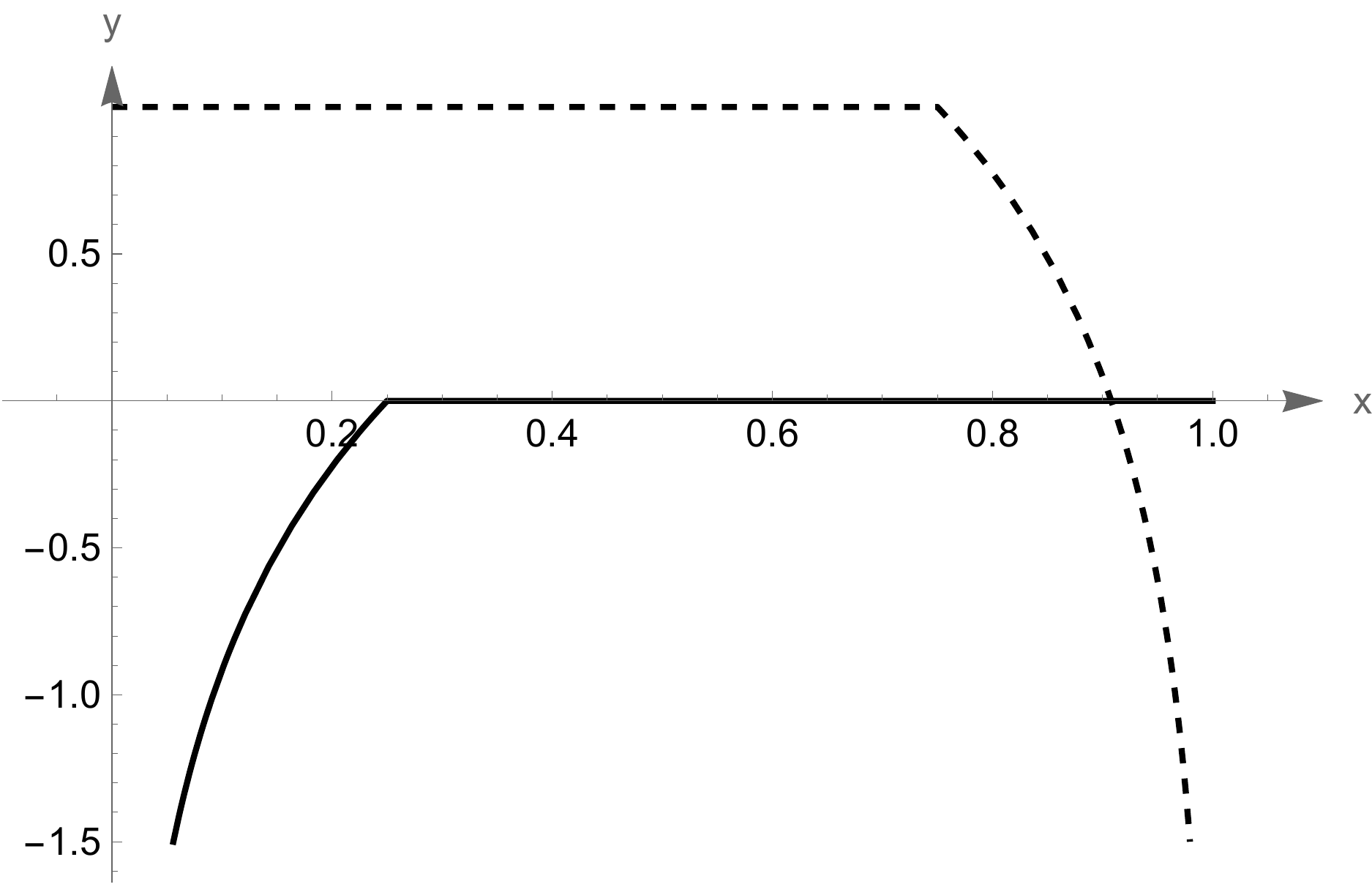}

\caption{The graphs of the interval maxima functions
 $m_0(x)$ (black) and $m_1(x)$ (black, dashed), 
see Example \ref{examp:2}.
}
\end{center}
\end{figure}
\end{example}

Indeed\footnote{For convenience, we write $F(x,t)$, $m_j(x)$, etc in place of
$F(\xx,t)$, $m_j(\xx)$, etc. respectively.},
$F(x,\cdot)$ cannot attain positive values on $[0,x]$.
If  $x\le b$, then $J(x)= 0$ and $L_a(x)\le 0$.
If $b\le x \le 1$, then by monotonicity of $L_a$,
$F(x,\cdot)$ is maximal on $[0,x]$ either at $b$ or at $0$.
The value at $b$ is $L_a(b-x)+1\le L_a(b-1)+1 <\log|(a/e)/a|+1=0$.
The value at $0$ is $F(x,0)\le 0$, moreover, taking into account that $x\ge a$
also holds, we have $L_a(0-x)=0$, so that $F(x,0)=0$.

It follows that $m_0(x)\le 0$.
Note that $m_0(0)=-\infty$ and $m_0$ is strictly increasing on $[0,a]$.
If $a\le x\le b$, then $F(x,t)$ is monotone decreasing on
$t\in[0,x]$, and $F(x,0)=0$ so  $m_0(x)= 0$.
If $b\le x \le 1$, then $F(x,t)$ is monotone decreasing on
$t\in[0,b]$ and is also monotone decreasing on $t\in[b,x]$.
Moreover $F(x,0)=0$ and $F(x,b)=1+L_a(b-x)\le 1+L_a(b-1)\le 1+\log|(1-b)/a|\le 0$,
so that $m_0(x)=F(x,0)=0$ in this case, too. In all, $m_0(x)\le 0$ for all $x$,  and $m_0(x)=0$ precisely for $a\le x\le 1$.

Regarding $m_1$, note that $J$ is monotone increasing and
$L_a(\cdot -x)$ is also monotone increasing on $[x,1]$,
so $m_1(x)=F(x,1)=1+L_a(1-x)$.
If $0\le x \le 1-a$, then $L_a(1-x)=0$,
so $m_1(x)=1$;
and if $1-a < x \le 1$, then $L_a(1-x)=\log((1-x)/a)$,
so $m_1(x)=\log((1-x)/a)+1 <1$.
In sum, $m_1(x)=1$ for $x\in [0,1-a]$ and
then it is strictly decreasing in $[1-a,1]$ from $1$ to $-\infty$, attaining $0$ exactly for $x=1-a/e$.

Comparing these cases for different ranges of $x$,
we see that $m_0(x)=m_1(x)$ holds if and only if
$m_0(x)=0$ and $m_1(x)=0$ which occurs precisely for $x=1-a/e$.

Therefore there is a unique equioscillation point.
Also, if $0\le x\le 1-a/e$, then $\underline{m}(x)=0$.

\begin{example}[{\bf Necessity of monotonicity}]\label{examp:mon}
Let
$$
K(t):=\min(\log|10t|,\log(\frac{10}{9}(1-|t|)))
$$
and $J(t):=0$.
Then $K$ is strictly concave, but not monotone.
Set $n=2$. Then
\begin{enumerate}[label={(\alph*)}]
\item
there are several equioscillating node systems, but their $\mol$ values are the same;
\item
there is a unique minimax node system;
\item
there are several maximin node systems; and
\item
strict majorization occurs.
\end{enumerate}
\end{example}

Observe first that $K(t)\le 0$ for $t\in[-1,1]$.
This immediately comes from that $\log|10 t|= \log(10(1-|t|)/9)$ holds if and only if $t=1/10$,
and $K(1/10)=0$ and $K$ is strictly monotone increasing on
$[0,1/10]$ and is strictly monotone decreasing on $[1/10,1]$.

Write the nodes as $\xx=(a-\delta,a+\delta)$, with $a:=\frac{x_1+x_2}{2}$ and $\delta:=\frac{x_2-x_1}{2}$.
By condition $0\le a-\delta\le a+\delta\le 1$,
in other words, $0\le a \le 1$ and $0\le \delta \le a,1-a$.
First, observe that
\begin{align*}
m_1(\xx)&=\sup\{F(\xx,t): a-\delta\le t \le a+\delta\}=F(\xx,a), \text{ so }
\\
m_1(\xx)&=\begin{cases}
2\log(10\delta)
& \text{ if } 0\le \delta\le 1/10,
\\
2\log\frac{10(1-\delta)}{9}
& \text{ if } 1/10\le \delta,
\end{cases}
\end{align*}
Also $m_1(\xx)\le 0$, and $m_1(\xx)= 0$
if and only if $\delta=1/10$.

If $t\in[a+\delta,1]$, then we have the following three cases depending
on $t$ ($a+\delta\le t \le a-\delta+0.1$ may be empty as well) with $u=t-a$, $\delta\le u \le 1-a$
\begin{equation*}
F(\xx,t)=
\begin{cases}
\log 100 (u-\delta)(u+\delta)
& \text{if } \delta\le 1/20, \delta\le u \le 1/10 -\delta,
\\
\log \frac{100}{9}(1-(u+\delta))(u-\delta)
& \text{if } \delta,1/10-\delta \le u \le \delta +1/10,
\\
\log \frac{100}{81}(1-(u+\delta))(1-(u-\delta))
& \text{if }  \delta +1/10 \le u \le 1-a.
\end{cases}
\end{equation*}
Here it is clear that the first expression is strictly increasing in $u$ and the third expression is strictly decreasing in $u$, so that $m_2(\xx)$ equals the maximum of the second expression.
Now elementary calculus shows that the second expression is strictly increasing in $u$ if $\delta+0.1 \le 1/2$, and if $\delta+0.1>1/2$, then it has strict local maximum at $u=1/2$, that is, when $t=a+1/2$.
Therefore
\[
m_2(\xx)=
\begin{cases}
F(\xx, a+\delta+1/10)=\log\left(1-\frac{20}{9}\delta\right)
& \text{ if } \delta+1/10\le 1/2,
\\
F(\xx, 1/2+a)= \log \frac{100}{9} \left(\frac{1}{2}-\delta\right)^2
& \text{ if } 4/10 \le \delta \le 1/2.
\end{cases}
\]

To determine the equioscillating configurations,
we compare the values of $m_1(\xx)$ and $m_2(\xx)$
in three cases depending on $\delta$.
They can be equal if and only if $\delta=\delta_0:=(\sqrt{82}-1)/90 \approx 0.0895$.
Therefore, any $\xx=(a-\delta_0, a+\delta_0)$, $a\in[\delta_0+1/10, 1-\delta_0-1/0]$ is an equioscillating configuration.

Also, $\overline{m}(\xx)\le 0$ and $\overline{m}(\xx)= 0$
if and only if $\xx=(a-\delta, a+\delta)$ and $\delta=0$
or $\delta=1/10$.

Moreover, both $m_1(\xx)$ and $m_2(\xx)$ are strictly
decreasing if $\delta\ge 1/10$ (and $\delta$ is not too large)
showing that strict majorization holds (for some configurations).

\begin{figure}[h]
\begin{center}
\includegraphics[keepaspectratio,width=0.4\textwidth]{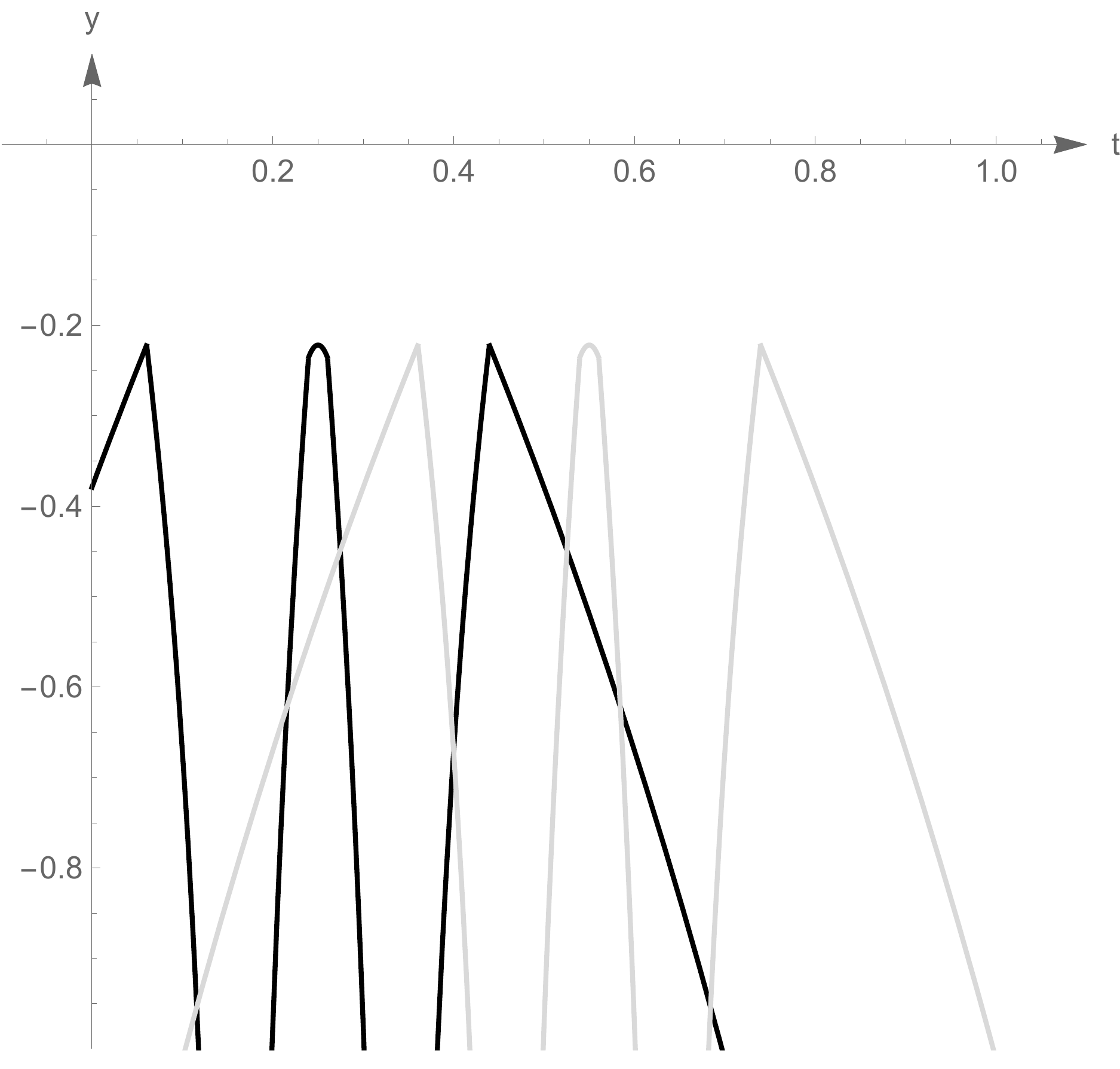}
\
\includegraphics[keepaspectratio,width=0.4\textwidth]{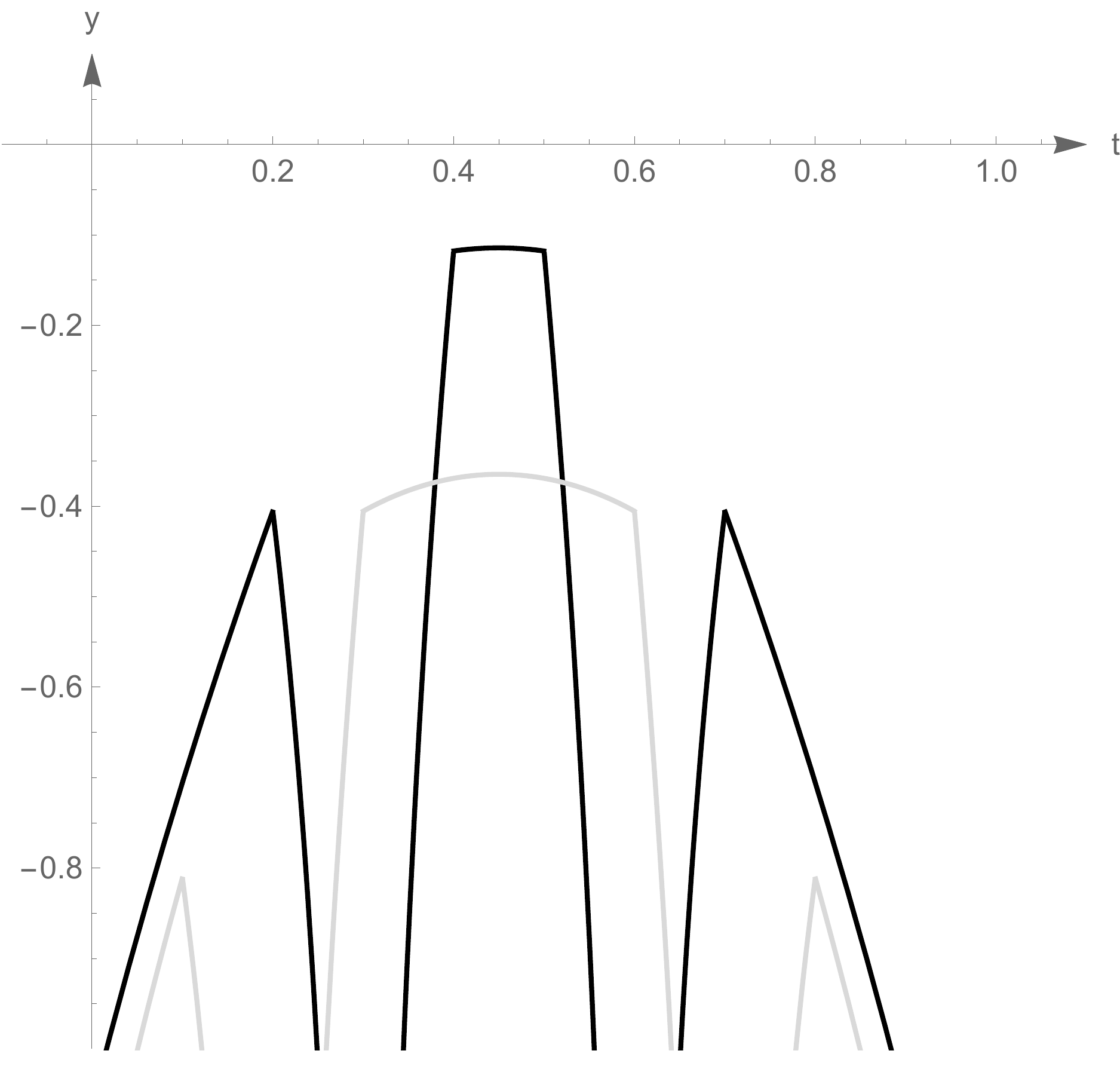}
\end{center}
\caption{
Graphs of $F(\xx,t)$ for $\xx=(a-\delta,a+\delta)$ on the left  for $a=1/4$, $\delta=(\sqrt{82}-1)/90$ (black) and $a=55/100$
 (grey); on the right for $a=45/100$, $\delta=15/100$ (black) and $\delta=25/100$ (grey), see Example \ref{examp:mon}.
}
\end{figure}

\section{Applications}
\label{sec:applicationsI}
\subsection{Bojanov's problem on the interval}
Consider now the set of \emph{monic} ``generalized algebraic polynomials'' (GAP, cf.{} Appendix A4, page 392 of \cite{ErdelyiBorwein}),
with given degree $\rr:=(r_1,\ldots,r_n)$, where $r_1,\ldots,r_n>0$ are given positive exponents:
\[
\PPrr[a,b]:=\left\{ P\ :\  P(t)=\prod_{j=1}^n |t-x_j|^{r_j} \ (t \in [a,b]),\
a\le x_1\le\ldots\le x_n\le b \right\}.
\]
Take an upper semicontinuous weight function $w:I\to [0,\infty)$, satisfying the condition
that it is non-zero at least at $n+1$ points
within the interval $I:=[a,b]$ 
(the endpoints $a,b$ are counted with weight $1/2$).
Consider the $w$-weighted uniform norm $\|\cdot\|_w$ defined
by $\|f\|_w:=\|fw\|_\infty:=\sup_I |f|w$.
Then Bojanov's Extremal Problem, extended to GAP-s, is to find the GAP $P\in \PPrr[a,b]$ with the least possible $\|P\|_w$.
If such an extremal polynomial exists, it will be called a \emph{Bojanov-Chebyshev polynomial}, so that $\|P\|_w=\min_{Q \in \PPrr[a,b]} \|Q\|_w$.

Actually, similarly to the classical 
case, there are two possible formulations of the extremal problem, 
since there is an ``unrestricted'' version, where we do not assume that the zeroes of 
the polynomial belong to $[a,b]$. However, 
here it is of importance that the order of the arising zero factors follow the order of the given exponents $r_j$, so for arbitrary complex zeros
the right interpretation is 
that we take
\[
\PPrr:=\left\{\prod_{j=1}^n |t-(x_j+iy_j)|^{r_j}\ :\  -\infty < x_1\le\ldots\le x_n<\infty, \ y_1,\ldots,y_n \in \RR \right\}.
\]
Thus the
\emph{(restricted)} Bo\-ja\-nov-Che\-by\-shev constant
is
$R_{\rr}^w[a,b]:=\min_{Q \in \PPrr[a,b]} \|Q\|_w$,
and the \emph{un\-res\-tric\-ted} Bo\-ja\-nov-Che\-by\-shev constant 
is
$C_{\rr}^w[a,b]:=\min_{Q \in \PPrr} \|Q\|_w$. As in the classical case, we easily see that
although formally $C_{\rr}^w[a,b]$ is an infimum over a larger set, we still have $C_{\rr}^w[a,b]=R_{\rr}^w[a,b]$, 
furthermore, extremizers exist only in $\PPrr[a,b]$ (if anywhere). 

We can now give a somewhat more precise statement than Theorem \ref{thm:BojanovGeneral}.

\begin{theorem}\label{thm:BojanovMoreGeneral} Let $n\in \NN$,  let $r_1,r_2,\ldots,r_n>0$ be positive numbers, 
let $[a,b]$ a non-de\-ge\-ne\-ra\-te, compact interval, and $w$ be an upper semicontinuous, 
non-negative weight function on $[a,b]$, assuming non-zero values at more than $n$  (weighted counting) points of the interval $[a,b]$.
Then $C_{\rr}^w[a,b]=R_{\rr}^w[a,b]$, and there exists one, unique 
Chebyshev-Bojanov extremal generalized  polynomial $P$, belonging to $\PPrr[a,b]$. 
This GAP has the form
\begin{equation}\label{Chebyshev-Bojanov-poly}
P(t)= \prod_{j=1}^n |t-x_j^*|^{r_j},
\end{equation}
with the node system $\xx^*:=(x_1^*,\ldots,x_n^*)$ satisfying $a<x_1^*<\ldots<x_n^*<b$ 
and uniquely determined by the following equioscillation property:
There exists an array of $n+1$ points 
$a \le t_0<t_1<t_2<\ldots<t_{n-1}<t_n\le b$ interlacing with the $x_i^*$, i.e., $a\le t_0<x_1^*<t_1<x_2^*<\ldots<x_n^*<t_n\le b$ 
such that
\begin{equation}\label{eq:osc}
P(t_k)w(t_k)=\|P\|_w \qquad (k=0,1,\ldots,n).
\end{equation}
Furthermore, if $w$ is in addition log-concave, then the unique Chebyshev-Bojanov extremal  generalized  polynomial $P$ is uniquely
determined by the property that there exists an array of $n+1$ points 
$a \le t_0<t_1<t_2<\ldots<t_{n-1}<t_n\le b$ such that \eqref{eq:osc} holds.
\end{theorem}

\begin{remark}\label{rem:GAP-BC} Note that Theorem 13.7 from \cite{TLMS2018} is the unweighted case. Also note that here we departed from considering the signatures, but in case $r_j\in \NN$ the analogous signed version can be seen easily. Also in fact one can assign signs to the factors of type $|t|^r$ arbitrarily, e.g. considering $|t|^r\sign t$ or, in case $r\in \NN$, $|t|^r (\sign t)^r$. Then the arising signed problem can be easily seen to become equivalent to the absolute value version. This shows that not the sign changes, but the attainment of minimal norm, are the decisive properties of an extremizer.
\end{remark}

\begin{proof}
By a simple linear substitution, it suffices to consider the case when $[a,b]=[0,1]$. Let $K:=\log|\cdot|$, $J:=\log w$. It is clear that $K$ is a singular, strictly concave, strictly monotone kernel function and $J$ is an $n$-field function. By taking logarithms, the original extremal problem of the minimization of $\|P\|_w$ is equivalent to the minimization problem $M(S)$ of $\sup F(\xx,\cdot)$ with the constants $r_j>0$, $(j=1,\ldots,n)$ as fixed in $\rr$.

We have already discussed that 
$C_{\rr}^w[0,1]=R_{\rr}^w[0,1]$. For the latter, we know $R_{\rr}^w[0,1]=\exp(M(S))$. 
Therefore, an application of Corollary \ref{cor:minimaxmaximin-spec} furnishes the characterization of extremal generalized polynomials.
To see the last assertion  note that by assumption $J$ is concave, so $F(\xx^*,\cdot)$ is strictly concave on each $I_j(\xx^*)$. 
By this the nodes and the (by strict concavity unique) maximum points $t_0,\ldots,t_n$ interlace, 
so the already proved characterization of extremal generalized polynomials applies.
\end{proof} 

Let us write $Q_\xx$ for the GAP with $Q_\xx(t):=\prod_{j=1}^n |t-x_j|^{r_j}$.
According to the above there is an even more precise understanding of the Bojanov-Chebyshev Problem. Writing $M_j(\xx):=\max_{\III_j(\xx)} |Qw|$,
we have the intertwining property that for \emph{any two admissible node systems} 
$\xx, \yy \in Y$  there exist indices $i,k$ with $M_i(\xx)<M_i(\yy)$ and $M_k(\xx)>M_k(\yy)$;
in particular, for any node system $\xx\ne \xx^*$ there exist indices $i,k$ with $M_i(\xx)<C_{\rr}^w[a,b]$ and $M_k(\xx)>C_{\rr}^w[a,b]$,
so that these Chebyshev constants are bounded from both sides
by interval maxima of an arbitrary node system:
$\min_{i=0,\ldots n} M_i(\xx) <C_{\rr}^w[a,b] < \min_{k=0,\ldots n} M_k(\xx)$.

Let us emphasize that the above discussion not only generalizes the Bojanov-Chebyshev Problem to weighted GAPs, but is much more general, given that we can take any log-concave, monotone factors $G(t)$ in place of $|t|$ (corresponding to $K(t):=\log G(t)$ being more general than $\log |t|$).

\subsection{Comparison of Chebyshev Constants of Union of Intervals}\label{sec:Widom}

The above discussion of various versions of the Chebyshev constant might have been considered trivial, but if we move to non-convex sets, then the distinction between restricted and non-restricted Chebyshev constants becomes essential.

Let $E \subset \RR$ be a compact set and $w\ge 0$ be a weight.
As above, define the \emph{restricted Bojanov-Chebyshev constant} $R_{\rr}^w(E):=\min_{Q \in \PPrr(E)} \|Q\|_w$, 
where
\[
\PPrr(E):=\left\{ P\ :\  P(t)=\prod_{j=1}^n |t-x_j|^{r_j} \ (t \in E),\
 x_1,\dots,x_n\in E\right\},
\]
and the \emph{unrestricted Bojanov-Chebyshev constant} $C_{\rr}^w(E):=\min_{Q \in \PPrr} \|Q\|_w$.

What we can do here is the following.

\begin{theorem}\label{thm:TwotypeChebyshev} Let $k, n\in \NN$, $a_1<b_1<a_2<b_2<\ldots<a_k<b_k$ be arbitrary real numbers, $E:=\cup_{\ell=1}^k [a_\ell,b_\ell]$, and $\rr\in (0,\infty)^n$ be an arbitrary exponent system. Then the restricted and unrestricted Chebyshev constants satisfy the inequality
\begin{equation}\label{ineqbetweenChebyshevs}
C_{\rr}^w(E) \le R_{\rr}^w(E) \le C(k,\rr) C_{\rr}^w(E),
\end{equation}
where $C(k,\rr):=2^{\max\{ r_{i_1}+\ldots+r_{i_{k-1}}\ :\  1\le i_1<\ldots<i_{k-1}\le n\}}$. In particular, if $\rr:={\bf 1}$, i.e., $\PPrr=\PP^1_n$, the family of the (absolute value of the) ordinary monic degree $n$ algebraic polynomials, then $C_n^w(E) \le R_n^w(E)\le 2^{k-1} C_n^w(E)$, independently of the value of $n$.
\end{theorem}
\begin{proof} As above, it is easy to see that for the unrestricted Chebyshev constant it suffices to consider polynomials with roots $x_j$
all in the closed convex hull.
$E^*:=\con E= [a_1,b_k]$ of $E$.
So by compactness of $E^*$ and upper semicontinuity of $w$ there exists an extremizer, and for this $P\in \PPrr(E^*)$ we have $C_\rr^w(E)=\|P\|_w(:=\max_E |Pw|)$. We will construct a $Q\in \PPrr(E)$ with $\|Q\|_w \le C(k,\rr) \|P\|_w$, so that minimizing over $\PPrr(E)$ would be seen to yield $R_\rr^w(E)\le C(k,\rr) \|P\|_w= C(k,\rr) C_\rr^w(E)$, as needed.

For convenience assume $E\subset [0,1]$, or for technical ease, even $a_1=0$ and $b_k=1$. 
Take $K(t):=\log|t|$, $J(t):=\log (w\chi_E)$, where $w$ is understood as defined all 
over $\RR$ and $\chi_E$ being the indicator function of $E$ (which is upper semicontinuous). 
Obviously, then we have for any $Q_\xx(t) \in \PPrr(E^*)$ with root system $\xx \in E^*=[0,1]$ that
\[
\log \|Q_{\xx}\|_w = \max_{t\in [0,1]} \Bigl(J(t)+\sum_{j=1}^n r_jK(t-x_j)\Bigr) =\max_{t\in [0,1]} F(\xx,t) = \mol(\xx).
\]
So in particular the minimality of $\|P\|_w$ over choices of zeroes $\xx \in E^*=[0,1]$ translates to the statement that
$P=Q_{\ww}$ with some $\ww \in \oS$ being a minimax point of $F$. As has been shown above in Theorem \ref{thm:EquiThm}, 
for the strictly concave and singular kernel function $K$, satisfying strict monotonicity
\eqref{cond:smonotone}, we have that $\ww \in Y$.
As in the complementary intervals $J_\ell:=(b_\ell,a_{\ell+1})\  (\ell=1,k-1)$ the indicator
function $\chi_E$ vanishes, also $J=-\infty$, and subintervals of that complementary intervals are singular.
So, no $\III_i(\ww)$ can be subinterval of those complementary intervals, because $\ww$ is non-singular.
In other words, in such a complementary interval 
there is at most one $w_i\in J_\ell$.

To construct our $Q_\xx$, i.e., the corresponding $\xx$ (with all $x_i\in E)$ and $F(\xx,\cdot)$, we choose the node system $\xx$ such that $x_i=w_i$ whenever $w_i\in E$, and $x_i=b_\ell$ or $x_i=a_{\ell+1}$, whichever is closer to $w_i$, in case $w_i\in J_\ell$ (and say $x_i:=b_\ell$ if they are of equal distance, i.e., $w_i=\frac{b_\ell+a_{\ell+1}}{2}$).

Let us compare the pure sum of translates functions for $\ww$ and $\xx$. We get
\[
f(\xx,t)-f(\ww,t)=\sum_{i\ w_i\not\in E} \left( r_iK(t-x_i)-r_iK(t-w_i) \right).
\]
If $i$ is such that $w_i\not\in E$, then $w_i\in J_\ell=(b_\ell,a_{\ell+1})$ for some $1\le \ell\le k-1)$, 
while for $t\in E$ we have either $t\le b_{\ell}$, or 
$t\geq a_{\ell+1}$. 
In case $x_i=b_{\ell}$ we have for all $0\le t\le b_\ell$ that $K(t-x_i)<K(t-w_i)$ by monotonicity. 
Let now $a_{\ell+1}\le t \le 1$. 
Then $K(t-x_i)=\log(t-x_i)=\log(t-w_i+(w_i-x_i))\le \log (2(t-w_i)) = \log2+K(t-w_i)$, for $t-w_i\ge a_{\ell+1}-w_i \ge \frac{a_{\ell+1}-b_\ell}{2}\ge w_i-x_i$ by construction. It follows that $K(t-x_i)\le \log 2 +K(t-w_i)$ for all $t\in E$. Similarly, it is easy to see that the same holds whenever $x_i=a_{\ell+1}$. Adding this for all indices $i$ with $w_i\not\in E$ we find
\[
f(\xx,t)-f(\ww,t)=\sum_{i\ w_i\not\in E} r_i \log2 \le \log C(k,\rr) \qquad (\forall t \in E).
\]
Therefore, we also have $F(\xx,t) \le \log C(k,\rr) + F(\ww,t)$
for all points $t\in [0,1]$, where $J(t)\ne -\infty$.
However, if $t\not\in E$, then adding $J(t)=-\infty$ makes both
sides $F(\xx,t)=F(\ww,t)=-\infty$, so the same inequality reads as $-\infty\le-\infty$,
and it remains in effect.
Finally, taking maxima we obtain $\mol(\xx)\le \mol(\ww)+\log C(k,\rr)$.
The assertion is proved.
\end{proof}

\section{Acknowledgment}

We are deeply indebted to V.V.~Arestov, V.I.~Berdyshev and M.V.~Deykalova
for providing us useful references and advice regarding classical literature on
Chebyshev type problems and partial maxima functions.
Also we are glad to mention the inspiring atmosphere of the regular Stechkin Summer Schools-Workshops,
which provided us an ideal forum to present and discuss our ever developing results with a generous and professional community.
We have benefited much from the comments and questions received there.

This research was partially supported by the DAAD-TKA
Research Project ``Harmonic Analysis and Extremal Problems'' \# 308015.

Szil\'ard Gy.~R\'ev\'esz was supported in part by Hungarian National Research,
Development and Innovation Fund project \# K-132097.


%

\begin{thebibliography}{10}\normalsize

\bibitem{AmbrusBallErdelyi}
G.~Ambrus, K.~M. Ball, and T.~Erd\'{e}lyi, \emph{Chebyshev constants for the
  unit circle}, Bull. Lond. Math. Soc. \textbf{45} (2013), no.~2, 236--248.

\bibitem{BermanPlemmons}
A.~Berman and R.~J. Plemmons, \emph{Nonnegative matrices in the mathematical
  sciences}, Classics in Applied Mathematics, vol.~9, Society for Industrial
  and Applied Mathematics (SIAM), Philadelphia, PA, 1994, Revised reprint of
  the 1979 original.

\bibitem{BojanovNaidenov1}
B.~D. Bojanov and N.~Naidenov, \emph{Exact {Markov}-type inequalities for
  oscillating perfect splines}, Constr. Approx. \textbf{18} (2002), no.~1,
  37--59 (English).

\bibitem{BojanovNaidenov2}
B.~D. Bojanov and N.~Naidenov, \emph{Alternation property and {Markov}'s
  inequality for {Tchebycheff} systems}, East J. Approx. \textbf{10} (2004),
  no.~4, 481--503 (English).

\bibitem{BojanovRahman}
B.~D. Bojanov and Q.~I. Rahman, \emph{On certain extremal problems for
  polynomials}, J. Math. Anal. Appl. \textbf{189} (1995), no.~3, 781--800
  (English).

\bibitem{Bojanov1979}
B.~D. Bojanov, \emph{A generalization of {C}hebyshev polynomials}, J. Approx.
  Theory \textbf{26} (1979), no.~4, 293--300.

\bibitem{ErdelyiBorwein}
P.~Borwein and T.~Erd\'{e}lyi, \emph{Polynomials and polynomial inequalities},
  Graduate Texts in Mathematics, vol. 161, Springer-Verlag, New York, 1995.

\bibitem{Davydov2}
O.~V. Davydov, \emph{A theorem on snakes for weak {D}escartes systems},
  Ukra\"{\i}n. Mat. Zh. \textbf{47} (1995), no.~3, 315--321.

\bibitem{Davydov}
O.~V. Davydov, \emph{A class of weak {C}hebyshev spaces and characterization of
  best approximations}, J. Approx. Theory \textbf{81} (1995), no.~2, 250--259.

\bibitem{Homeo}
B.~Farkas, B.~Nagy, and {\relax Sz}.~{\relax Gy}. R\'ev\'esz, \emph{On a
  homeomorphism theorem related to sum of translates functions},
  arXiv:2112.11029.

\bibitem{TLMS2018}
B.~Farkas, B.~Nagy, and {\relax Sz}.~{\relax Gy}. R\'{e}v\'{e}sz, \emph{A
  minimax problem for sums of translates on the torus}, Trans. London Math.
  Soc. \textbf{5} (2018), no.~1, 1--46.

\bibitem{FentonEnt}
P.~C. Fenton, \emph{The minimum of small entire functions}, Proceedings of the
  American Mathematical Society \textbf{81} (1981), no.~4, 557--561.

\bibitem{FentonCos}
P.~C. Fenton, \emph{{$\cos\pi\lambda$} again}, Proc. Amer. Math. Soc.
  \textbf{131} (2003), no.~6, 1875--1880.

\bibitem{FentonCos2}
P.~C. Fenton, \emph{A refined {$\cos\pi\rho$} theorem}, J. Math. Anal. Appl.
  \textbf{311} (2005), no.~2, 675--682.

\bibitem{Fenton}
P.~C. Fenton, \emph{A min-max theorem for sums of translates of a function}, J.
  Math. Anal. Appl. \textbf{244} (2000), no.~1, 214--222.

\bibitem{Goldberg}
A.~A. Goldberg, \emph{The minimum modulus of a meromorphic function of slow
  growth}, Mat. Zametki \textbf{25} (1979), no.~6, 835--844, 956.

\bibitem{Haar}
A.~Haar, \emph{Die {M}inkowskische {G}eometrie und die {A}nn\"{a}herung an
  stetige {F}unktionen}, Math. Ann. \textbf{78} (1917), no.~1, 294--311.

\bibitem{HardinKendallSaff}
D.~P. Hardin, A.~P. Kendall, and E.~B. Saff, \emph{Polarization optimality of
  equally spaced points on the circle for discrete potentials}, Discrete
  Comput. Geom. \textbf{50} (2013), no.~1, 236--243.

\bibitem{Karlin}
S.~Karlin, \emph{Representation theorems for positive functions}, J. Math.
  Mech. \textbf{12} (1963), 599--617.

\bibitem{KarlinStudden}
S.~Karlin and W.~J. Studden, \emph{Tchebycheff systems: {W}ith applications in
  analysis and statistics}, Pure and Applied Mathematics, Vol. XV, Interscience
  Publishers John Wiley \& Sons, New York-London-Sydney, 1966.

\bibitem{NikolovShadrin}
G.~Nikolov and A.~Shadrin, \emph{On {M}arkov-{D}uffin-{S}chaeffer inequalities
  with a majorant}, Constructive theory of functions, Prof. M. Drinov Acad.
  Publ. House, Sofia, 2012, pp.~227--264.

\bibitem{NikolovShadrin2}
G.~Nikolov and A.~Shadrin, \emph{On {M}arkov-{D}uffin-{S}chaeffer inequalities
  with a majorant. {II}}, 2014, pp.~175--197.

\bibitem{Nikolov}
G.~P. Nikolov, \emph{Snake polynomials and {M}arkov-type inequalities},
  Approximation theory, DARBA, Sofia, 2002, pp.~342--352.

\bibitem{Parthasarathy}
T.~Parthasarathy, \emph{On global univalence theorems}, Lecture Notes in
  Mathematics, vol. 977, Springer-Verlag, Berlin-New York, 1983.

\bibitem{Rankin}
R.~A. Rankin, \emph{On the closest packing of spheres in {$n$} dimensions},
  Ann. of Math. (2) \textbf{48} (1947), 1062--1081.

\end{thebibliography}
\providecommand{\bysame}{\leavevmode\hbox to3em{\hrulefill}\thinspace}
\providecommand{\MR}{\relax\ifhmode\unskip\space\fi MR }
\providecommand{\MRhref}[2]{%
  \href{http://www.ams.org/mathscinet-getitem?mr=#1}{#2}
}
\providecommand{\href}[2]{#2}

\end{document}